\documentclass[11pt]{article}
\usepackage{amssymb,amsmath,amsthm}
\usepackage[mathscr]{eucal}
\usepackage[cm]{fullpage}
\usepackage[english]{babel}
\usepackage[latin1]{inputenc}
\usepackage{float}
\usepackage{tikz}
\usepackage{xcolor}
\usepackage{mathrsfs}

\newtheorem{theorem}{Theorem}
\newtheorem{lemma}[theorem]{Lemma}
\newtheorem{proposition}[theorem]{Proposition}
\newtheorem{corollary}[theorem]{Corollary}

\newtheorem{conjecture}{Conjecture}

\def\im{\mathop{\mathrm{Im}}\nolimits} 
\def\dom{\mathop{\mathrm{Dom}}\nolimits}
\def\ker{\mathop{\mathrm{Ker}}\nolimits}
\def\PAP{\mathrm{PAP}}
\def\A{\mathrm{A}}
\def\B{\mathrm{B}}
\def\T{\mathrm{T}}
\def\I{\mathrm{I}}
\def\S{\mathrm{S}}
\def\C{\mathrm{C}}
\def\Z{\mathrm{Z}}
\def\D{\mathrm{D}}
\def\O{\mathrm{O}}
\def\M{\mathrm{M}}
\def\OP{\mathrm{OP}}
\def\OR{\mathrm{OR}}
\def\r{\mathbf{r}}
\def\rank{\mathrm{rank}}
\def\o{\mathbf{o}}
\def\X{\mathscr{X}}

\def\setdif{{\smallsetminus}}
\def\N{\mathbb{N}}

\def\smallt#1{\left(\begin{smallmatrix}#1\end{smallmatrix}\right)}
\def\transf#1{\left(\begin{matrix}#1\end{matrix}\right)}
\def\inv{\mathrm{inv}}

\title{%
Groups of permutations that are even on maximal proper subsets, and related monoids
}

\author{
\textbf{V\'{\i}tor Hugo Fernandes}\\
Center for Mathematics and Applications (NOVA Math) 
and Department of Mathematics,  \\ 
Faculdade de Ci\^{e}ncias e Tecnologia, 
Universidade Nova de Lisboa, 
2829-516 Caparica, 
Portugal\\ 
E-mail: vhf@fct.unl.pt\\
ORCID iD: https://orcid.org/0000-0003-1057-4975\\ 
}

\begin{document}

\maketitle

\begin{abstract}
Let $n$ be a positive integer and let $[n]=\{1,2,\ldots,n\}$. Let $\Gamma_n$ denote the group of permutations on $[n]$ whose restrictions to maximal proper subsets of $[n]$ are even, let $\Sigma_n$ denote the monoid of transformations on $[n]$ whose injective restrictions to maximal proper subsets of $[n]$ are even and let $\Delta_n$ denote the submonoid of $\Sigma_n$ generated by transformations of rank at least $n-1$. 
In this paper, we present descriptions of $\Gamma_n$, $\Delta_n$ and $\Sigma_n$, determine their cardinalities and ranks, and provide minimal generating sets for each of them.
\end{abstract}


\medskip 


\section{Introduction} \label{sec:intro} 

For a positive integer $n$, let us denote the set $\{1,2,\ldots,n\}$ by $[n]$. 
Let $\T_n$ be the monoid of all full transformations on $[n]$, 
$\I_n$ the inverse monoid of all partial permutations (i.e. injective partial transformations) on $[n]$ 
and $\S_n$ the symmetric group on $[n]$. 
If $\alpha$ is a partial transformation on $[n]$, we denote the domain and image of $\alpha$ by $\dom(\alpha)$ and $\im(\alpha)$, respectively. 
Let $\alpha\in\I_n$. By an inversion of $\alpha$ we mean a pair $\{a,b\}\subseteq\dom(\alpha)$ such that $a<b$ and $a\alpha>b\alpha$. 
The set of inversions of $\alpha$ is denoted by $\inv(\alpha)$. We say that $\alpha$ is even [respectively, odd]  if $|\inv(\alpha)|$ is even [respectively, odd]. 
Obviously, this notion consistently extends to partial permutations the usual definition of even/odd permutations. 
Let us now recall a lemma proved in \cite{Fernandes&Vernitski:2026}, which we will also use in this paper. 

\begin{lemma}[{\cite[Lemma 2]{Fernandes&Vernitski:2026}}] \label{lema2} 
Let $\alpha,\beta\in\I_n$ be such that $\dom(\beta)=\im(\alpha)$. If $\alpha$ and $\beta$ are both even or both odd then $\alpha\beta$ is even. 
If one of $\alpha$ and $\beta$ is even and the other is odd then $\alpha\beta$ is odd.
\end{lemma}

For $1\leqslant t\leqslant n$, define $\Gamma_n^t$ as the set consisting of those permutations on $[n]$ whose every restriction of width $t$ (i.e. to a $t$-element subset of $[n]$) is even. It is clear, due to Lemma \ref{lema2}, that the set $\Gamma_n^t$ is a subgroup of $\S_n$. Moreover, $\Gamma_n^1=\S_n$ and $\Gamma_n^n$ is the alternating group $\A_n$ on $[n]$. 
Let us consider the following permutations of $[n]$: 
the identity transformation $\iota_n$ on $[n]$;
the $n$-cycle $\sigma_n=\smallt{1&2&\cdots&n-1&n\\2&3&\cdots&n&1}$; and 
the reflexion $\rho_n=\smallt{1&2&\cdots&n-1&n\\n&n-1&\cdots&2&1}$. 
Along with the two-row notation for transformations, for convenience, we will also use in this paper the cycle notation for permutations. 
For instance, $\sigma_n=(1,2,\ldots,n)$. 
Let $\Z_2$ be the $2$-element group $\{\iota_n,\rho_n\}$, $\C_n$ the cyclic group of order $n$ generated by $\sigma_n$ and   
$\D_{2n}$ the dihedral group of order $2n$ generated by $\sigma_n$ and $\rho_n$. 

Interestingly, the groups $\Gamma_2,\ldots,\Gamma_{n-2}$ exhibit the following periodic behavior \cite[Theorem 7]{Fernandes&Vernitski:2026}: 
for $2\leqslant t \leqslant n-2$, 
if $t = 2  \,(\mathrm{mod} \ 4)$, then $\Gamma_n^t$ is the trivial group $\{\iota_n\}$;
if $t = 3  \,(\mathrm{mod} \, 4)$, then $\Gamma_n^t=\C_n$;
if $t = 0  \,(\mathrm{mod} \, 4)$, then $\Gamma_n^t=\Z_2$; and 
if $t = 1  \,(\mathrm{mod} \, 4)$, then $\Gamma_n^t=\D_{2n}$. 

On the other hand, the situation of the group $\Gamma_n^{n-1}$ is quite different. 
This case is related to the group $\PAP_n$ of so-called parity-alternating permutations of $[n]$, i.e. permutations $\sigma\in\S_n$ such that 
$a\sigma$ and $(a+1)\sigma$ have different parities for all $1\leqslant a\leqslant n-1$ \cite{tanimoto2010combinatorics,tanimoto2010parity}. 
Observe that, $\PAP_n=\{\sigma\in\S_n\mid  [n]_\mathrm{o}\sigma= [n]_\mathrm{o}\}=\{\sigma\in\S_n\mid  [n]_\mathrm{e}\sigma= [n]_\mathrm{e}\}$, if $n$ is odd, and 
$\PAP_n=\{\sigma\in\S_n\mid\mbox{$ [n]_\mathrm{o}\sigma= [n]_\mathrm{o}$ or $ [n]_\mathrm{o}\sigma= [n]_\mathrm{e}$}\}=\{\sigma\in\S_n\mid\mbox{$ [n]_\mathrm{e}\sigma= [n]_\mathrm{o}$ or $ [n]_\mathrm{e}\sigma= [n]_\mathrm{e}$}\}$, if $n$ is even, 
where $ [n]_\mathrm{o}=\{a\in[n]\mid\mbox{$a$ is odd}\}$ and $ [n]_\mathrm{e}=\{a\in[n]\mid\mbox{$a$ is even}\}$. 
It is clear that the size of $\PAP_n$ is $2(\frac{n}{2})!^2$, if $n$ is even, and $(\frac{n-1}{2})!(\frac{n+1}{2})!$, if $n$ is odd. 

Let $X_a=[n]\setdif\{a\}$ for $a\in[n]$. For each subset $X$ of $[n]$ and $\alpha\in\T_n$, denote the restriction of $\alpha$ to $X$ by $\alpha_{|X}$. 
So, $\Gamma_n^{n-1}$ consists of all permutations $\sigma$ on $[n]$ such that $\sigma_{|X_a}$ is even for all $a \in [n]$. 
Let us consider the \textit{opposite} set $\text{-}\Gamma_n^{n-1}$ of $\Gamma_n^{n-1}$ 
of all  permutations $\sigma$ on $[n]$ such that $\sigma_{|X_a}$ is odd for all $a \in [n]$.
Then, $\PAP_n$ is the disjoint union of $\Gamma_n^{n-1}$ and $\text{-}\Gamma_n^{n-1}$ \cite[Lemma 10]{Fernandes&Vernitski:2026} 
(notice that, this means that, if $\sigma\in\PAP_n$, then $\sigma\in\Gamma_n^{n-1}$ if and only if $\sigma_{|X_a}$ is even for \textit{some} $a\in[n]$) and so, 
in particular, $\Gamma_n^{n-1}\subseteq\PAP_n$; 
$|\Gamma_n^{n-1}| = \frac{1}{2} |\PAP_n|$ for $n \geqslant 3$ \cite[Theorem 11]{Fernandes&Vernitski:2026}; and 
if $n$ is odd, then $\Gamma_n^{n-1} = \PAP_n \cap \A_n$ \cite[Proposition 12]{Fernandes&Vernitski:2026}. 
If $n$ is even, a description of $\Gamma_n^{n-1}$ will be presented in Section \ref{gamma} of this paper. 

We would like to point out that the groups $\Gamma_n^t$, along with the properties we have just recalled here, 
were first presented by A. Vernitski in an early version of the paper \cite{Fernandes&Vernitski:2026}.

Now, for each $1\leqslant t\leqslant n$, let us consider the submonoid $\Sigma_n^t$ of $\T_n$ consisting of those transformations on $[n]$ 
whose all injective restrictions of width $t$ are even. 
Before reviewing the properties of these monoids studied in \cite{Fernandes&Vernitski:2026}, we need to introduce some notation. 

Denote the rank of a transformation $\alpha\in\T_n$, i.e. the size of its image set, by $\r(\alpha)$. 
Notice that, in turn, we denote the \textit{rank} of a semigroup, monoid, or group $S$ (i.e., its minimum number of generators) by $\rank(S)$. 
For a subset $S$ of $\T_n$ and $1\leqslant k\leqslant n$, let $S(\r=k)$ [respectively, $S(\r  \leqslant k)$, $S(\r  \geqslant k)$] be 
the subset of $S$ consisting of those mappings whose rank is equal to $k$ [respectively, is less than or equal to $k$, is greater than or equal to $k$]. 

From the definition, it is immediately clear that $\Sigma_n^1=\T_n$, $\Sigma_n^n=\T_n(\r  \leqslant n-1)\cup \A_n$, 
$\Gamma_n^t=\Sigma_n^t\cap \S_n$, i.e. $\Gamma_n^t$ is the group of units of $\Sigma_n^t$, for $1\leqslant t\leqslant n$, and  
$\Sigma_n^t(\r  \leqslant t-1) = \T_n(\r  \leqslant t-1)$ for $2\leqslant t\leqslant n$ \cite[Proposition 13]{Fernandes&Vernitski:2026}. 
What is not quite as obvious is the way these monoids embed into one another: for $2\leqslant p,q\leqslant n$,  
$\Sigma_n^p\subseteq\Sigma_n^q$ if and only if $p\leqslant q$ and 
at least one of the following conditions is true, $p = 2  \,(\mathrm{mod} \ 4)$ or $q = 1  \,(\mathrm{mod} \, 4)$ or $p = q  \, (\mathrm{mod} \, 4)$
\cite[Theorem 20]{Fernandes&Vernitski:2026}. 

Next, recall that a transformation $\alpha\in\T_n$ is said to be order-preserving 
[respectively, order-reversing] if $a\leqslant b$ implies $a\alpha\leqslant b\alpha$
[respectively, $a\alpha\geqslant b\alpha$] for all $a,b \in [n]$, 
and it is said to be orientation-preserving [respectively, orientation-reversing]
if there exists no more than one $a\in [n]$ such that
$a\alpha >(a+1)\alpha$ [respectively, $a\alpha<(a+1)\alpha$],
where $n+1$ denotes $1$ here. 
A transformation on $[n]$ that is order-preserving or order-reversing 
[respectively, orientation-preserving or orientation-reversing] is also said to be monotone [respectively, oriented]. 
Let us denote by $\O_n$ [respectively, $\M_n$, $\OP_n$ and $\OR_n$]  
the submonoid of $\T_n$ 
of all order-preserving [respectively, monotone, order-preserving, oriented] transformations. 
Semigroups of monotone, order-preserving, orientation-preserving and oriented transformations 
have been studied massively for more than six decades. 
See, for example, \cite{AR,Aizenstat:1962,Aizenstat:1962b,Catarino&Higgins:1999,
Fernandes:1997,Fernandes:2002a,fernandes2023CKMS,Fernandes&Gomes&Jesus:2006,Fernandes&al:2010,Fernandes&Volkov:2010,
Gomes&Howie:1992,Higgins:1995,higgins2022orientation,Howie:1971,Laradji&Umar:2006,McA,Vernitskii&Volkov:1995}. 
 As we will see below, these monoids are closely related to the monoids $\Sigma_n^2,\ldots,\Sigma_n^{n-2}$.

In fact, the equalities $\Sigma_n^2=\O_n$ \cite[Proposition 14]{Fernandes&Vernitski:2026} and $\Sigma_n^3 = \T_n(\r  \leqslant 2) \cup \OP_n$ \cite[Proposition 17]{Fernandes&Vernitski:2026} give us the closest relationships. In addition, for $1\leqslant t\leqslant n$, we have $\O_n\subseteq\Sigma_n^t$, and 
for $2\leqslant t\leqslant n-2$, we have: $\OP_n\subseteq\Sigma_n^t$, if $t$ is odd; $\M_n\subseteq\Sigma_n^t$, if $t = 0  \,(\mathrm{mod} \, 4)$ or 
$t = 1  \,(\mathrm{mod} \, 4)$; and $\OR_n\subseteq\Sigma_n^t$, if $t = 1  \,(\mathrm{mod} \, 4)$ \cite[Proposition 16]{Fernandes&Vernitski:2026}. 
Furthermore, for $2\leqslant t\leqslant n-2$, we also have \cite[Proposition 21]{Fernandes&Vernitski:2026}: 
$\Sigma_n^t(\r  \geqslant t+2)  = \O_n(\r  \geqslant t+2)$, if $t = 2  \,(\mathrm{mod} \, 4)$; 
$\Sigma_n^t(\r  \geqslant t+2) = \OP_n(\r  \geqslant t+2)$, if $t = 3  \,(\mathrm{mod} \, 4)$;  
$\Sigma_n^t(\r  \geqslant t+2) = \M_n(\r  \geqslant t+2)$, if $t = 0  \,(\mathrm{mod} \, 4)$; and 
$\Sigma_n^t(\r  \geqslant t+2) = \OR_n(\r  \geqslant t+2)$, if $t = 1  \,(\mathrm{mod} \, 4)$. 
Notice that, for $2\leqslant t\leqslant n-2$, describing sets $\Sigma_n^t(\r  = t)$ and $\Sigma_n^t(\r  = t+1)$ remains an open problem. 
Nevertheless, the properties listed above led the authors to consider the following alternative family of monoids. 

For $1\leqslant t\leqslant n$, let $\Delta_n^t$ be the monoid generated by $\Sigma_n^t(\r  \geqslant n-1)$. Clearly, as well as for $\Sigma_n^t$, 
the group of units of the  monoid $\Delta_n^t$ is the group $\Gamma_n^t$, for $1\leqslant t\leqslant n$. 
We have $\Delta_n^1=\Sigma_n^1=\T_n$ (since $\T_n$ is generated by its transformations of rank greater than or equal to $n-1$) 
and $\Delta_n^n=\Sigma_n^n=\T_n(\r  \leqslant n-1)\cup \A_n$. 
But the most interesting fact is that the monoids $\Delta_n^2,\ldots,\Delta_n^{n-3}$ exhibit periodic behavior similar to that of the groups $\Gamma_n^2,\ldots,\Gamma_n^{n-2}$. Indeed, for $2\leqslant t\leqslant n-3$, we have \cite[Theorem 22]{Fernandes&Vernitski:2026}: 
$\Delta_n^t=\O_n$, if $t = 2  \,(\mathrm{mod} \, 4)$;  
$\Delta_n^t=\OP_n$, if $t = 3  \,(\mathrm{mod} \, 4)$;  
$\Delta_n^t=\M_n$, if $t = 0  \,(\mathrm{mod} \, 4)$; and 
 $\Delta_n^t=\OR_n$, if $t = 1  \,(\mathrm{mod} \, 4)$. 

Notice that, we have not mentioned anything non-immediate about the monoids $\Delta_n^{n-2}$ and $\Delta_n^{n-1}$, 
nor about the monoids $\Sigma_n^{n-2}$ and $\Sigma_n^{n-1}$. 
In fact, the descriptions of these monoids were left as an open problem in \cite{Fernandes&Vernitski:2026}. 

This paper is specifically devoted to the monoids $\Delta_n^{n-1}$ and $\Sigma_n^{n-1}$, as well as their group of units $\Gamma_n^{n-1}$. 
In Section \ref{gamma}, we begin by describing the group $\Gamma_n^{n-1}$ when $n$ is even, 
then proceed to obtain a set of minimal generators (i.e. a generating set of minimum size) for $\Gamma_n^{n-1}$ and determine its rank. 
In Section \ref{deltasigma}, we present complete descriptions of $\Delta_n^{n-1}$ and $\Sigma_n^{n-1}$ and calculate their cardinalities. 
Finally, in Section \ref{rankdeltasigma}, we present a minimal generating set and determine the ranks of the monoids $\Delta_n^{n-1}$ and $\Sigma_n^{n-1}$. 

From now on, for convenience, we denote $\Gamma_n^{n-1}$, $\Delta_n^{n-1}$ and $\Sigma_n^{n-1}$ simply by  $\Gamma_n$, $\Delta_n$ and $\Sigma_n$, respectively. 

It is easy to see that 
\begin{itemize}

\item $\Gamma_1=\PAP_1=\Delta_1=\Sigma_1=\T_1=\{\iota_1\}$;

\item $\Gamma_2=\PAP_2=\S_2$ and $\Delta_2=\Sigma_2=\T_2$; 

\item $\PAP_3=\left\{\iota_3,\smallt{1&2&3\\3&2&1}\right\}$, $\Gamma_3=\{\iota_3\}$ and 
\begin{align*}
\Delta_3=\Sigma_3=\left\{ 
\iota_3, \smallt{1&2&3\\1&2&2}, \smallt{1&2&3\\1&3&3}, \smallt{1&2&3\\2&3&3},
\smallt{1&2&3\\1&1&2}, \smallt{1&2&3\\1&1&3}, \smallt{1&2&3\\2&3&3}, 
\smallt{1&2&3\\1&1&1}, \smallt{1&2&3\\2&2&2}, \smallt{1&2&3\\3&3&3}
\right\}. 
\end{align*}
We have $|\Delta_3|=10$ and (as a monoid) $\Delta_3$ has rank $3$, 
where for exemple $\left\{\smallt{1&2&3\\1&2&2}, \smallt{1&2&3\\1&1&3}, \smallt{1&2&3\\2&3&3}\right\}$ is a minimal generating set. 
This statement can be easily verified through extensive calculations using an appropriate computational tool. 
\end{itemize}
So, from this point forward, we will focus our attention primarily on $n\geqslant4$. 

For general background on Semigroup
Theory and standard notation, we refer the reader to Howie's book \cite{Howie:1995}. 
We would also like to mention the use of computational tools, namely GAP \cite{GAP} and its package \cite{sgpviz}.

\section{The groups $\Gamma_n$}\label{gamma}

In this section, our main objectives are to provide a description of $\Gamma_n$, 
particularly for the previously missing case of an even $n$, 
to discuss minimal generating sets, and to determine the rank of $\Gamma_n$.
For convenience, we will introduce and study some properties of other subgroups of $\S_n$.  

Let 
$$
\PAP_n^+=\{\sigma\in\S_n\mid  [n]_\mathrm{o}\sigma= [n]_\mathrm{o}\}=\{\sigma\in\PAP_n\mid 1\sigma\in [n]_\mathrm{o}\}
$$
and
$$
\PAP_n^-=\{\sigma\in\S_n\mid [n]_\mathrm{o}\sigma= [n]_\mathrm{e}\}=\{\sigma\in\PAP_n\mid 1\sigma\in [n]_\mathrm{e}\}.
$$
Then, obviously, $\PAP_n$ is a disjoint union of $\PAP_n^+$ and $\PAP_n^-$. In addition, 
for an odd $n$, $\PAP_n^+(=\PAP_n$, since $\PAP_n^-=\emptyset$) is isomorphic to $\S_{\frac{n+1}{2}}\times\S_{\frac{n-1}{2}}$ and, 
for an even $n$, $\PAP_n^+$ is a subgroup of $\PAP_n$ isomorphic to $\S_{\frac{n}{2}}\times\S_{\frac{n}{2}}$. 
Let $\Gamma_n^+=\PAP_n^+\cap\A_n$ and $\Gamma_n^-=\PAP_n^-\cap(\S_n\setdif\A_n)$. 
Observe that, for an odd $n$, as $\PAP_n^-=\emptyset$, we also have $\Gamma_n^-=\emptyset$. 
The following result completes and generalizes \cite[Proposition 12]{Fernandes&Vernitski:2026}. 

\begin{theorem}\label{gama_n}
For any positive integer $n$, $\Gamma_n=\{\sigma\in\PAP_n\mid \sigma\in\A_n\Leftrightarrow \mbox{$1\sigma$ is odd}\}$ 
and is a disjoint union of $\Gamma_n^+$ and $\Gamma_n^-$.  
\end{theorem} 
\begin{proof}
If $n$ is odd, then the result follows immediately from \cite[Proposition 12]{Fernandes&Vernitski:2026}. So, let us suppose that $n$ is even. 
For $\sigma\in\S_n$, define $\hat{\sigma}\in\S_{n+1}$ by $i\hat{\sigma}=i\sigma$ for $i\in[n]$ and $(n+1)\hat{\sigma}=n+1$.  
Let $\sigma\in\PAP_n^+$. Then, clearly, $\hat{\sigma}\in\PAP_{n+1}$, $\inv(\hat{\sigma})=\inv(\sigma)$ and 
$\inv(\hat{\sigma}_{|X_i\cup\{n+1\}})=\inv(\sigma_{|X_i})$ for any $i\in[n]$. 
One the other hand, as $n+1$ is odd, $\Gamma_{n+1}=\PAP_{n+1}\cap\A_{n+1}$. 
Hence, 
$$
\sigma\in\A_n\Longleftrightarrow \hat{\sigma}\in\A_{n+1} \Longleftrightarrow \hat{\sigma}\in\Gamma_{n+1}
\Longleftrightarrow \mbox{$\hat{\sigma}_{|X_1\cup\{n+1\}}$ is even} \Longleftrightarrow \mbox{$\sigma_{|X_1}$ is even} 
\Longleftrightarrow \sigma\in\Gamma_n. 
$$

Now, suppose that $\sigma\in\PAP_n^-$. Let $\pi=(1,2)(3,4)\cdots(n-1,n)$. Then, $\pi\in\PAP_n^-$, $\pi$ is a composition of $\frac{n}{2}$ transpositions and 
$\pi_{|X_1}$ has $\frac{n}{2}-1$ inversions.  Hence, $\pi\in\A_n$ if and only if $n=0\,(\mathrm{mod}\,4)$, and $\pi\in\Gamma_n$ if and only if $n=2\,(\mathrm{mod}\,4)$.
On the other hand, $\sigma\pi\in\PAP_n^+$ and so, as we have shown above, 
$$
\begin{array}{rl}
\sigma\pi\in\Gamma_n  &\Longleftrightarrow \sigma\pi\in\A_n \Longleftrightarrow \mbox{$\sigma,\pi\in\A_n$ or $\sigma,\pi\not\in\A_n$} \\ 
& \Longleftrightarrow \mbox{($\sigma\in\A_n$ and $n=0\,(\mathrm{mod}\,4)$) or ($\sigma\not\in\A_n$ and $n=2\,(\mathrm{mod}\,4)$).}
\end{array}
$$
Let us consider each of the two cases for $n$ mentioned just behind. 

\noindent{\sc case 1}. $n=0\,(\mathrm{mod}\,4)$. In this case, we have $\pi\in\A_n$ and $\pi\not\in\Gamma_n$. 
If $\sigma\in\A_n$, then $\sigma\pi\in\Gamma_n$ and so $\sigma\not\in\Gamma_n$. 
Conversely, suppose that $\sigma\not\in\Gamma_n$. Then, $\sigma_{|X_1}$ is odd. 
Since $\pi_{|X_1\sigma}$ is also odd, by \ref{lema2},  
$(\sigma\pi)_{|X_1}=\sigma_{|X_1}\pi_{|X_1\sigma}$ is even, whence $\sigma\pi\in\Gamma_n$ and so $\sigma\in\A_n$. 

\noindent{\sc case 2}. $n=2\,(\mathrm{mod}\,4)$. Now, we have $\pi\not\in\A_n$ and $\pi\in\Gamma_n$. 
If $\sigma\in\Gamma_n$, then $\sigma\pi\in\Gamma_n$ and so $\sigma\not\in\A_n$. 
Conversely, suppose that $\sigma\not\in\A_n$. Then, $\sigma\pi\in\Gamma_n$ and so $\sigma=(\sigma\pi)\pi^{-1}\in\Gamma_n$. 

Thus, for any $\sigma\in\PAP_n^-$, we have $\sigma\in\Gamma_n$ if and only if $\sigma\not\in\A_n$, which concludes the proof. 
\end{proof} 

\bigskip 

Although our focus is on the groups $\Gamma_n$, it will be useful to consider the groups we will define below. 

In what follows, it is convenient to adopt the following notations: 
if $X$ is a non-empty set, then $\S(X)$ and $\A(X)$ denote the symmetric group and the alternating group on $X$, respectively. 
For a non-empty subset $Y$ of $X$, it is also convenient to identify the elements of $\S(Y)$ with the elements of its natural injection into $\S(X)$. 
Let $[n']=\{1',2',\ldots,n'\}$ be a set of $n$-elements (i.e. a copy of $[n]$). 
Let $m,n\geqslant2$ and define 
$$
\S_{m\oplus n}=\{\sigma\in\S([m]\cup[n'])\mid[m]\sigma=[m]\}
=\{\sigma\in\S([m]\cup[n'])\mid[n']\sigma=[n']\}.
$$ 
Then, $\S_{m\oplus n}$ is a subgroup of $\S([m]\cup[n'])$ isomorphic to the direct product $\S_m\times\S_n$. 
Define also 
$$
\Gamma_{m\oplus n}=\S_{m\oplus n}\cap\A([m]\cup[n']).
$$
Given the identification we established earlier, it is clear that 
\begin{equation}\label{gmn} 
\Gamma_{m\oplus n}=\A([m])\A([n'])\cup(\S([m])\setdif\A([m]))(\S([n'])\setdif\A([n'])), 
\end{equation}
$\A([m]),\A([n'])\subseteq\Gamma_{m\oplus n}$ and $|\Gamma_{m\oplus n}|=\frac{m!n!}{2}$. 
Moreover, $\Gamma_{m\oplus n}$ and $\Gamma_{n\oplus m}$ (just like $\S_{m\oplus n}$ and $\S_{n\oplus m}$) are naturally isomorphic groups. 
Furthermore,  it is clear that, for $n\geqslant4$, if $n$ is odd, then $\Gamma_n$ is isomorphic to $\Gamma_{\frac{n+1}{2}\oplus\frac{n-1}{2}}$ and,  
if $n$ is even, then $\Gamma_n^+$ is isomorphic to $\Gamma_{\frac{n}{2}\oplus\frac{n}{2}}$. 

Recall that $\S_n$ is generated by $\{(1,2),(1,2,\ldots,n)\}$ and also by  $\{(1,2),(2,3,\ldots,n)\}$. 
Let $\lambda_0=(1,2)\in\S([m])$, $\lambda'_0=(1',2')\in\S([n'])$ and 
$\lambda=\lambda_0\lambda'_0$. Then, $\lambda\in\Gamma_{m\oplus n}$. 
Further, let 
$$
\mu=\left\{\begin{array}{ll}
(1,2,\ldots,m) & \mbox{if $m$ is odd}\\
(2,3,\ldots,m) & \mbox{if $m$ is even}\\
\end{array}\right.
~\mbox{and}~
\mu'=\left\{\begin{array}{ll}
(1',2',\ldots,n') & \mbox{if $n$ is odd}\\
(2',3',\ldots,n') & \mbox{if $n$ is even.}\\
\end{array}\right. 
$$
Then, $\mu\in\A([m])$ and $\mu'\in\A([n'])$, whence $\mu,\mu'\in\Gamma_{m\oplus n}$. 

Let us denote the order of a permutation $\sigma$ by $\o(\sigma)$. 

\begin{proposition}\label{gamamn}
Let $m,n\geqslant2$. Then, $\Gamma_{m\oplus n}=\langle\lambda,\mu,\mu'\rangle$. 
\end{proposition}
\begin{proof}
Let $\sigma\in\Gamma_{m\oplus n}$. 
It suffices to show that $\sigma\in \langle\lambda,\mu,\mu'\rangle$. In view of (\ref{gmn}), 
$\sigma=\theta\theta'$ for some ($\theta\in\A([m])$ and $\theta'\in\A([n'])$) or 
($\theta\in\S([m])\setdif\A([m])$ and $\theta'\in\S([n'])\setdif\A([n'])$). 
In particular, $\theta\in\S([m])=\langle \lambda_0,\mu\rangle$ and $\theta'\in\S([n'])=\langle \lambda'_0,\mu'\rangle$. 
Hence, $\theta=\lambda_0^{r_1}\mu^{k_1}\lambda_0\mu^{k_2}\lambda_0\cdots\lambda_0\mu^{k_r}\lambda_0^{r_2}$ for some 
$r\geqslant0$, $1\leqslant k_1,\ldots,k_r\leqslant\o(\mu)$ and $0\leqslant r_1,r_2\leqslant 1$. 
Observe that, as $\mu\in\A([m])$, then $\theta\in\A([m])$ if and only if $r_1+r_2+r-1$ is even. 
Hence, 
$$
\lambda^{r_1}\mu^{k_1}\lambda\mu^{k_2}\lambda\cdots\lambda\mu^{k_r}\lambda^{r_2} = \theta{\lambda'_0}^{r_1+r_2+r-1}
= \left\{\begin{array}{ll} 
\theta & \mbox{if $\theta\in\A([m])$}\\ 
\theta\lambda'_0 & \mbox{if $\theta\in\S([m])\setdif\A([m])$}, 
\end{array}\right.
$$
since $\lambda_0\lambda'_0=\lambda'_0\lambda_0$ and $\lambda'_0\mu=\mu\lambda'_0$. 
In a similar manner, 
$\theta'={\lambda'_0}^{r'_1}{\mu'}^{k'_1}{\lambda'_0}{\mu'}^{k'_2}{\lambda'_0}\cdots{\lambda'_0}{\mu'}^{k'_{r'}}{\lambda'_0}^{r'_2}$ for some 
$r'\geqslant0$, $1\leqslant k'_1,\ldots,k'_{r'}\leqslant\o(\mu')$ and $0\leqslant r'_1,r'_2\leqslant 1$, and so 
$$
{\lambda}^{r'_1}{\mu'}^{k'_1}{\lambda}{\mu'}^{k'_2}{\lambda}\cdots{\lambda}{\mu'}^{k'_{r'}}{\lambda}^{r'_2} 
= {\lambda_0}^{r'_1+r'_2+r'-1}\theta' 
= \left\{\begin{array}{ll} 
\theta' & \mbox{if $\theta'\in\A([n'])$}\\ 
\lambda_0\theta' & \mbox{if $\theta'\in\S([n'])\setdif\A([n'])$}, 
\end{array}\right.
$$
since we also have $\lambda_0\mu'=\mu'\lambda_0$. Thus, 
if $\theta\in\A([m])$ and $\theta'\in\A([n'])$, then 
$$
\sigma=\lambda^{r_1}\mu^{k_1}\lambda\mu^{k_2}\lambda\cdots\lambda\mu^{k_r}\lambda^{r_2} 
{\lambda}^{r'_1}{\mu'}^{k'_1}{\lambda}{\mu'}^{k'_2}{\lambda}\cdots{\lambda}{\mu'}^{k'_{r'}}{\lambda}^{r'_2} \in \langle\lambda,\mu,\mu'\rangle,  
$$
and if $\theta\in\S([m])\setdif\A([m])$ and $\theta'\in\S([n'])\setdif\A([n'])$, then 
$$
\begin{array}{rcl}
\sigma &= &\lambda^{r_1}\mu^{k_1}\lambda\mu^{k_2}\lambda\cdots\lambda\mu^{k_r}\lambda^{r_2} \lambda'_0\lambda_0 
{\lambda}^{r'_1}{\mu'}^{k'_1}{\lambda}{\mu'}^{k'_2}{\lambda}\cdots{\lambda}{\mu'}^{k'_{r'}}{\lambda}^{r'_2} \\
&=& 
\lambda^{r_1}\mu^{k_1}\lambda\mu^{k_2}\lambda\cdots\lambda\mu^{k_r}\lambda^{r_2+r'_1+1} 
{\mu'}^{k'_1}{\lambda}{\mu'}^{k'_2}{\lambda}\cdots{\lambda}{\mu'}^{k'_{r'}}{\lambda}^{r'_2} \in \langle\lambda,\mu,\mu'\rangle, 
\end{array} 
$$
as required. 
\end{proof} 

Notice that $\Gamma_{2\oplus2}=\{\iota,(1,2)(1',2')\}=\langle(1,2)(1',2')\rangle$, where $\iota$ denotes the identity of $\Gamma_{2\oplus2}$. 
So, $\Gamma_{2\oplus2}$ has rank $1$. On the other hand, if $m>2$ or $n>2$, then $\rank(\Gamma_{m\oplus n})\geqslant2$. 
In fact, let $\sigma=\theta\theta'\in \Gamma_{m\oplus n}$, with $\theta\in\S([m])$ and $\theta'\in\S([n'])$. Since $\theta\theta'=\theta'\theta$, 
any element of $\langle\sigma\rangle$ is of the form $\theta^k{\theta'}^k$, for some $k\in\N$. If $\Gamma_{m\oplus n}=\langle\sigma\rangle$, then 
$\lambda_0,\mu\in\langle\theta\rangle$ and $\lambda'_0,\mu'\in\langle\theta'\rangle$, 
since $\lambda,\mu,\mu'\in\Gamma_{m\oplus n}$, and so $\langle\theta\rangle=\S([m])$ and $\langle\theta'\rangle=\S([n'])$, 
whence $m=n=2$. 

As $\mu\mu'=\mu'\mu$, if $\gcd(\o(\mu),\o(\mu'))=1$, then $\langle\mu,\mu'\rangle=\langle\mu\mu'\rangle$. So, we have the following consequence of the previous result. 

\begin{corollary}\label{gcdmn} 
Let $m,n\geqslant 2$ be such $\gcd(\o(\mu),\o(\mu'))=1$. Then, $\Gamma_{m\oplus n}=\langle\lambda,\mu,\mu'\rangle$. 
Moreover, in this case, if $m>2$ or $n>2$, then $\rank(\Gamma_{m\oplus n})=2$. 
\end{corollary} 

Notice that
$$
\o(\mu)=\left\{\begin{array}{ll}
m & \mbox{if $m$ is odd}\\
m-1 &  \mbox{if $m$ is even}
\end{array}\right. 
~\mbox{and}~
\o(\mu')=\left\{\begin{array}{ll}
n & \mbox{if $n$ is odd}\\
n-1 &  \mbox{if $n$ is even.}
\end{array}\right. 
$$

Using GAP \cite{GAP}, we were able to verify through extensive calculations that $\Gamma_{3\oplus3}$, $\Gamma_{4\oplus3}$, and $\Gamma_{4\oplus4}$ are not generated by any two of their elements, hence 
$\rank(\Gamma_{3\oplus3})=\rank(\Gamma_{4\oplus3})=\rank(\Gamma_{4\oplus4})=3$,  taking into account Proposition \ref{gamamn}. 
On the other hand, in addition to the examples arising from Corollary \ref{gcdmn} and Theorem \ref{gamman} below, all the examples we have thoroughly checked using GAP (for example, $\Gamma_{5\oplus5}$, $\Gamma_{6\oplus6}$, $\Gamma_{7\oplus7}$, $\Gamma_{8\oplus8}$, $\Gamma_{9\oplus3}$, $\Gamma_{9\oplus4}$, etc.) have rank $2$. 
Thus, these results lead us to the following conjecture, which we leave as an open problem. 

\begin{conjecture}
Let $m\geqslant n\geqslant 2$ be such that $(m,n)\not\in\{(2,2),(3,3),(4,3),(4,4)\}$. Then, $\rank(\Gamma_{m\oplus n})=2$. 
\end{conjecture} 

We now turn our attention once again to the groups $\Gamma_n$. For $n\geqslant4$, consider the following permutations on $[n]$: 
$$
\theta_n=(1,3)(2,4),~ \eta_n=(3,5,\ldots,n)(2,4,\ldots,n-1) ~ \mbox{and}~ \zeta_n=(1,3,\ldots,n)(4,6,\ldots,n-1). 
$$
Observe that, if $n$ is odd, then $\theta_n,\eta_n,\zeta_n\in\Gamma_n$ and, if $n$ is even, then $\theta_n,\sigma_n\in\Gamma_n$.  
Also, notice that $\eta_n\theta_n=\zeta_n$ and $\zeta_n\theta_n=\eta_n$. 

Let us start by assuming that $n$ is odd. Hence, as previously noted, $\Gamma_n$ is isomorphic to $\Gamma_{\frac{n+1}{2}\oplus\frac{n-1}{2}}$. 
If $n=1\,(\mathrm{mod}\,4)$, then $\frac{n+1}{2}$ is odd and $\frac{n-1}{2}$ is even, whence 
$$
\Gamma_n=\langle (1,3)(2,4),(1,3,\ldots,n),(4,6,\ldots,n-1)\rangle,
$$
by Proposition \ref{gamamn}. Moreover, in this case, $\gcd(\frac{n+1}{2},\frac{n-3}{2})=1$ and so, by Corollary \ref{gcdmn}, 
\begin{equation}\label{1mod4}
\Gamma_n=\langle \theta_n,\zeta_n\rangle = \langle \theta_n,\eta_n\rangle. 
\end{equation}
On the other hand, suppose that $n=3\,(\mathrm{mod}\,4)$. Then, $\frac{n+1}{2}$ is even and $\frac{n-1}{2}$ is odd, whence 
$$
\Gamma_n=\langle (1,3)(2,4),(3,5,\ldots,n),(2,4,\ldots,n-1)\rangle,
$$
by Proposition \ref{gamamn}. 
We have $\zeta_n^{\frac{n-3}{2}}=(1,3,\ldots,n)^{\frac{n-3}{2}}=(1,3,\ldots,n)^{-2}$, whence 
$\eta_n\zeta_n^{\frac{n-3}{2}}\eta_n=(1,n,3)(2,4,\ldots,n-1)^2$ and so 
$$
(\eta_n\zeta_n^{\frac{n-3}{2}}\eta_n)^{\frac{n-3}{4}}=(1,n,3)^{\frac{n-3}{4}}(2,4,\ldots,n-1)^{\frac{n-3}{2}}=(1,n,3)^{\frac{n-3}{4}}(2,4,\ldots,n-1)^{-1}. 
$$
Since 
$$
(1,n,3)^{\frac{n-3}{4}}=\left\{\begin{array}{ll}
\iota_n & \mbox{if $n=3\,(\mathrm{mod}\,12)$}\\
(1,n,3) & \mbox{if $n=7\,(\mathrm{mod}\,12)$}\\
(1,3,n) & \mbox{if $n=11\,(\mathrm{mod}\,12)$,}
\end{array}\right. 
$$
we have 
$$
(\eta_n\zeta_n^{\frac{n-3}{2}}\eta_n)^{\frac{n-3}{4}}\eta_n = \left\{\begin{array}{ll}
(3,5,\ldots,n) & \mbox{if $n=3\,(\mathrm{mod}\,12)$}\\
(1,n,3)(3,5,\ldots,n) & \mbox{if $n=7\,(\mathrm{mod}\,12)$}\\
(1,3,n)(3,5,\ldots,n) & \mbox{if $n=11\,(\mathrm{mod}\,12)$.}
\end{array}\right. 
$$
In addition, if $n=11\,(\mathrm{mod}\,12)$, then $(3,5,\ldots,n)=\theta_n (\eta_n\zeta_n^{\frac{n-3}{2}}\eta_n)^{\frac{n-3}{4}}\eta_n \theta_n$ and, 
if $n=7\,(\mathrm{mod}\,12)$, then 
$$
\eta_n^4 \nu_n^{\frac{n-5}{2}} \eta_n^{\frac{n-5}{2}} \nu_n^3 \eta_n^{\frac{n-5}{2}} \nu_n^{\frac{n-5}{2}} = 
\eta_n^4 \nu_n^{-1} \eta_n^{-2} \nu_n^3 \eta_n^{-2} \nu_n^{-1} = 
\left\{\begin{array}{ll}
\iota_n & \mbox{if $n=7$}\\
(3,5,\ldots,n) & \mbox{if $n>7$,}
\end{array}\right. 
$$
where $\nu_n= (\eta_n\zeta_n^{\frac{n-3}{2}}\eta_n)^{\frac{n-3}{4}}\eta_n =(1,3)(5,7,\ldots,n)$. 
Hence, for $n>7$ we showed that $(3,5,\ldots,n)\in\langle \theta_n,\eta_n\rangle$, 
whence also $(2,4,\ldots,n-1)=(3,5,\ldots,n)^{-1}\eta_n\in \langle \theta_n,\eta_n\rangle$ and so 
\begin{equation}\label{3mod4}
\Gamma_n=\langle \theta_n,\eta_n\rangle=\langle \theta_n,\zeta_n\rangle.
\end{equation} 
Recall that, we already noticed that $\Gamma_{4\oplus3}$ has rank $3$, so the same holds for $\Gamma_7$.  

Next, we suppose that $n$ is even. In this case, we have that $\Gamma_n^+$ is isomorphic to $\Gamma_{\frac{n}{2}\oplus\frac{n}{2}}$. 
Let 
$$
\mu_n=\left\{\begin{array}{ll}
(3,5,\ldots,n-1) & \mbox{if $n=0\,(\mathrm{mod}\,4)$}\\
(1,3,\ldots,n-1) & \mbox{if $n=2\,(\mathrm{mod}\,4)$}\\
\end{array}\right.
~\mbox{and}~
\mu'_n=\left\{\begin{array}{ll}
(4,6,\ldots,n) & \mbox{if $n=0\,(\mathrm{mod}\,4)$}\\
(2,4,\ldots,n) & \mbox{if $n=2\,(\mathrm{mod}\,4)$}\\
\end{array}\right. 
$$
Then, by Proposition \ref{gamamn}, we have 
$$
\Gamma_n^+=\langle \theta_n,\mu_n,\mu'_n\rangle. 
$$
\begin{lemma}\label{gama+}
If $n$ is even and $\Gamma_n^+\subseteq\langle\theta_n,\sigma_n\rangle$, then $\Gamma_n=\langle\theta_n,\sigma_n\rangle$.
\end{lemma}
\begin{proof}
First, notice that, since $\theta_n,\sigma_n\in\Gamma_n$, we have $\langle\theta_n,\sigma_n\rangle\subseteq\Gamma_n$. 
On the other hand, to prove the converse inclusion, given the hypothesis and Theorem \ref{gama_n}, 
it suffices to show that $\Gamma_n^-\subseteq \langle\theta_n,\sigma_n\rangle$. 
So, let us take $\sigma\in\Gamma_n^-$. Then, $\sigma\sigma_n\in\Gamma_n^+$, whence $\sigma\sigma_n\in \langle\theta_n,\sigma_n\rangle$, 
by hypothesis, and so $\sigma=(\sigma\sigma_n)\sigma_n^{-1}\in\langle\theta_n,\sigma_n\rangle$, as required. 
\end{proof} 

For $n=4$, we have $\Gamma_4^+=\{\iota_4,\theta_4\}=\langle\theta_4\rangle\subseteq\langle\theta_4,\sigma_4\rangle=\langle\sigma_4\rangle$, 
whence $\Gamma_4=\langle\sigma_4\rangle$. So, let us suppose that $n\geqslant6$ (and $n$ is even). 
If $n=2\,(\mathrm{mod}\,4)$, then 
$$
\mu_n=(\sigma_n^{-1}(\theta_n\sigma_n\theta_n\sigma_n^{n-3})^{\frac{n-4}{2}}\sigma_n^{n-3})^{\frac{n-2}{4}}
~\mbox{and}~ \mu'_n=\mu_n^{-1}\sigma_n^2,
$$
whence $\mu_n,\mu'_n\in \langle\theta_n,\sigma_n\rangle$, and so $\Gamma_n^+\subseteq\langle\theta_n,\sigma_n\rangle$. 
Alternatively, if $n=0\,(\mathrm{mod}\,4)$, then 
$$
\mu_n=((\sigma_n^{n-3}\theta_n\sigma_n\theta_n)^{\frac{n-6}{2}}\sigma_n^{n-6})^{\frac{n-4}{4}}
~\mbox{and}~ \mu'_n=\mu_n^{-1}\sigma_n^2\theta_n,
$$
whence $\mu_n,\mu'_n\in \langle\theta_n,\sigma_n\rangle$, and so once again $\Gamma_n^+\subseteq\langle\theta_n,\sigma_n\rangle$. 
Therefore, in both cases, by Lemma \ref{gama+}, we get
\begin{equation}\label{evenn}
\Gamma_n=\langle\theta_n,\sigma_n\rangle.
\end{equation} 

Our next theorem summarizes what we have just proven in (\ref{1mod4}), (\ref{3mod4}) and (\ref{evenn}). 

\begin{theorem}\label{gamman}
For $n=5$ and $n\geqslant9$ when $n$ is odd, $\Gamma_n=\langle \theta_n,\eta_n\rangle=\langle \theta_n,\zeta_n\rangle$
and $\Gamma_n$ has rank $2$. 
For $n\geqslant6$ when $n$ is even, $\Gamma_n=\langle \theta_n,\sigma_n\rangle$ and 
$\Gamma_n$ has rank $2$. In addition,  $\Gamma_4$ is the cyclic group of order $4$ generated by $\sigma_4$,  
$\Gamma_7=\langle (1,3)(2,4), (3,5,7), (2,4,6)\rangle$ and $\Gamma_7$ has rank $3$. 

\end{theorem} 

\section{The monoids $\Delta_n$ and $\Sigma_n$}\label{deltasigma}

Throughout this section, we provide complete descriptions of the elements of the monoids $\Delta_n$ and $\Sigma_n$.  
In addition, we determine the cardinalities of these monoids. 

For each $\alpha\in\T_n(\r=n-1)$, let us denote by $[\alpha]$ the only non-trivial kernel class of $\alpha$.  

\begin{proposition}\label{n-1}
Let $n\geqslant3$. Let $\alpha\in\T_n(\r=n-1)$ and let $[\alpha]=\{x,y\}$. 
\begin{enumerate}
\item If $\alpha\in\Sigma_n$, then $x+y$ is odd. 
\item  If $x+y$ is odd, then: 
\begin{enumerate}
\item $\alpha\in\Sigma_n$ if and only if $\alpha_{|X_x}$ is even; 
\item $\alpha\in\Sigma_n$ if and only if  $\alpha_{|X_y}$ is even;  
\item Exactly half of the elements of the $\mathscr{R}$-class of $\alpha$ in $\T_n$ belong to $\Sigma_n$. 
\end{enumerate}
\end{enumerate}
\end{proposition} 
\begin{proof}
Let us suppose that $x<y$ and let $A=\{x+1,\ldots,y-1\}$. Notice that $\alpha$ has two transversals, namely $X_x$ and $X_y$. 
Hence, for each $a\in A$, $\alpha_{|X_x}$ has one more or one less inversion than $\alpha_{|X_y}$, and so 
if $|A|$ is even (i.e. $x+y$ is odd), then $\alpha_{|X_x}$ is even if and only if $\alpha_{|X_y}$ is even, and if $|A|$ is odd (i.e. $x+y$ is even), 
then $\alpha_{|X_x}$ is even if and only if $\alpha_{|X_y}$ is odd. 
This immediately proves \textit{1}, \textit{2(a)} and \textit{2(b)}. 

As for \textit{2(c)}, let us consider the $\mathscr{H}$-class $H_x$ contained in the $\mathscr{R}$-class of $\alpha$ of the elements that have $X_x$ as their image. 
Then, clearly, $H_x$ is a group isomorphic to $\S_{n-1}$, 
from which it follows that exactly half of its restrictions to $X_x$ are even. Therefore, by \textit{2(a)},  
exactly half of the elements of $H_x$ belong to $\Sigma_n$.
Now, let $a\in X_x$ and let $H_a$ be the $\mathscr{H}$-class contained in the $\mathscr{R}$-class of $\alpha$ of the elements that have $X_a$ as their image. 
Let 
$$
\varepsilon=\smallt{1&\cdots&x-1&x&x+1&\cdots&n\\1&\cdots&x-1&y&x+1&\cdots&n}, 
\quad
\beta=\left\{\begin{array}{ll}
\smallt{1&\cdots&a-1&a&a+1&\cdots&x-1&x&x+1&\cdots&n\\1&\cdots&a-1&a+1&a+2&\cdots&x&x&x+1&\cdots&n} &\mbox{if $a<x$}\\ [3mm] 
\smallt{1&\cdots&x&x+1&\cdots&a-1&a&a+1&\cdots&n\\1&\cdots&x&x&\cdots&a-2&a-1&a+1&\cdots&n} &\mbox{if $x<a$}
\end{array}\right. 
$$
and $\gamma=\varepsilon\beta$. 
Then, it is easy to see that $\varepsilon\in H_x\cap\Sigma_n$, $\beta\in\Sigma_n(\r=n-1)$ and $\gamma\in H_a\cap\Sigma_n$. 
Furthermore, by Green's Lemma, the mapping $H_x\longrightarrow H_a$, $\xi\longmapsto\xi\beta$, is a bijection. 
Let $\xi\in H_x$. Then, $X_x\xi=X_x$, whence $(\xi\beta)_{|X_x}=\xi_{|X_x}\beta_{|X_x}$ and so, as $\beta_{|X_x}$ is even, 
we have, by \textit{2(a)} and Lemma \ref{lema2}, 
$$
\xi\beta\in\Sigma_n\Longleftrightarrow \mbox{$(\xi\beta)_{|X_x}$ is even}  \Longleftrightarrow \mbox{$\xi_{|X_x}$ is even}   \Longleftrightarrow \xi\in\Sigma_n.
$$
Hence, $H_x\cap\Sigma_n\longrightarrow H_a\cap\Sigma_n$, $\xi\longmapsto\xi\beta$, is also a bijection, from which follows that 
exactly half of the elements of $H_a$ belong to $\Sigma_n$ as well. Therefore, 
exactly half of the elements of the $\mathscr{R}$-class of $\alpha$ in $\T_n$ belong to $\Sigma_n$. 
\end{proof} 

Since $\Sigma_n$ has $\Gamma_n$ as its group of units and $\T_n(\r\leqslant n-2)\subseteq\Sigma_n$, the previous result completes the description of $\Sigma_n$. 
Now, as we have $\frac{n-1}{2}\frac{n+1}{2}$ pairs of $[n]$ such that $x+y$ is odd, if $n$ is odd, and $(\frac{n}{2})^2$ pairs of $[n]$ such that $x+y$ is odd, if $n$ is even, then 
\begin{equation}
|\Sigma_n(\r=n-1)| = \left\{
\begin{array}{ll}
\frac{n!(n-1)(n+1)}{8} & \mbox{if $n$ is odd}\\   
\frac{n!n^2}{8} & \mbox{if $n$ is even}   
\end{array}
\right.
\end{equation} 
for $n\geqslant3$, by Proposition \ref{n-1}. 
On the other hand, from $|\Sigma_n(\r\leqslant n-2)|=|\T_n(\r\leqslant n-2)|= n^n  - (1+ \binom{n}{2})n!$  
and $\Sigma_n(\r=n)=\Gamma_n$, 
it follows that 
\begin{equation}\label{sigmasize} \hspace*{-1mm}
|\Sigma_n| = \left\{
\begin{array}{ll}
n^n-(1+\binom{n}{2})n!+  \frac{1}{2}(\frac{n-1}{2})!(\frac{n+1}{2})!  +  \frac{n!(n-1)(n+1)}{8} & \mbox{if $n$ is odd}\\ [1mm] 
n^n-(1+\binom{n}{2})n!+ (\frac{n}{2})!^2  +  \frac{n!n^2}{8} & \mbox{if $n$ is even} 
\end{array}
\right.
\end{equation} 
for $n\geqslant3$. 

Now, let $\X=\{\alpha\in\T_n(\r\leqslant n-2)\mid\mbox{if $x\alpha=y\alpha$, then $x+y$ is even, for all $x,y\in [n]$}\}$. 

\begin{theorem}\label{delta}
For $n\geqslant3$, $\Delta_n=\Sigma_n\setdif\X$, i.e. 
$\Delta_n$ consists of all transformations of $\Sigma_n$ that are either permutations or have at least one kernel class containing elements of distinct parities. 
\end{theorem}

In fact, this theorem also holds for $n=1$ and $n=2$, since in these cases we can assume that $\T_n(\r\leqslant n-2)=\emptyset$. 

Observe that,  if $\alpha\in\Sigma_n(\r\leqslant n-1)$, then 
$\alpha\in\Sigma_n\setdif\Delta_n$ if and only if each kernel class of $\alpha$ consists only of even numbers or only of odd numbers. 

Before proving Theorem \ref{delta}, we present two lemmas. 

\begin{lemma}\label{delta1} 
Let $n\geqslant4$ and let $\alpha\in\T_n(\r\leqslant n-2)\setdif\X$. Then, there exist $\alpha_1,\alpha_2\in\T_n\setdif\X$ such that $\alpha=\alpha_1\alpha_2$, 
$\r(\alpha_1)=\r(\alpha)+1$ and $\r(\alpha_2)=n-2$. In particular, if $\r(\alpha)\leqslant n-3$, then 
$\alpha_1,\alpha_2\in\Sigma_n\setdif\X$ and $\r(\alpha)<\r(\alpha_1),\r(\alpha_2)$. 
\end{lemma}
\begin{proof}
Since $\alpha\not\in\X$, there exist $x,y\in[n]$ such that $x\alpha=y\alpha$ and $x+y$ is odd. On the other hand, since $\r(\alpha)\leqslant n-2$, there exist  
$u\in[n]\setdif\{x,y\}$ such that the kernel class $(u\alpha)\alpha^{-1}$ is non-trivial (notice that we may or may not have $u\alpha=x\alpha$) 
and $b_1,b_2\in[n]\setdif\im(\alpha)$ such that $b_1<b_2$. 
Define $\alpha_1\in\T_n$ by 
$$
i\alpha_1=\left\{ \begin{array}{ll} 
b_1 & \mbox{if $i=u$}\\
i\alpha & \mbox{otherwise}. 
\end{array}\right. 
$$
Then, $\im(\alpha_1)=\im(\alpha)\cup\{b_1\}$ and $x\alpha_1=y\alpha_1$, whence $\r(\alpha_1)=\r(\alpha)+1$ and $\alpha_1\not\in\X$. 

If $b_1<b_2-1$, then define $\alpha_2\in\T_n$ by 
$$
j\alpha_2=\left\{ \begin{array}{ll} 
u\alpha & \mbox{if $j=b_1$}\\
b_2-1 & \mbox{if $j=b_2$}\\
j & \mbox{otherwise}. 
\end{array}\right. 
$$
Hence, $\im(\alpha_2)=[n]\setdif\{b_1,b_2\}$ and so $\r(\alpha_2)=n-2$. Moreover, as $(b_2-1)\alpha_2=b_2-1=b_2\alpha_2$, it follows that $\alpha_2\not\in\X$. 
On the other hand, if $b_1=b_2-1$, then define $\alpha_2\in\T_n$ by 
$$
j\alpha_2=\left\{ \begin{array}{ll} 
u\alpha & \mbox{if $j=b_1,b_2$}\\
j & \mbox{otherwise}. 
\end{array}\right. 
$$
Also, in this case, $\im(\alpha_2)=[n]\setdif\{b_1,b_2\}$, whence $\r(\alpha_2)=n-2$, and $(b_2-1)\alpha_2=u\alpha=b_2\alpha_2$, so $\alpha_2\not\in\X$. 
Clearly, in both cases, we have $\alpha=\alpha_1\alpha_2$. 

Finally, observe that, if $\r(\alpha)\leqslant n-3$, then $\r(\alpha_1),\r(\alpha_2)\leqslant n-2$, 
and so $\alpha_1,\alpha_2\in\T_n\setdif\X$ implies $\alpha_1,\alpha_2\in\Sigma_n\setdif\X$, as required. 
\end{proof} 

\begin{lemma}\label{delta2} 
Let $n\geqslant4$ and let $\alpha\in\T_n(\r\leqslant n-1)$ be such that $x\alpha=y\alpha$ for some $x,y\in[n]$ with $x+y$ an odd number. 
Then, there exist $\sigma,\tau\in\Gamma_n$ and $c\in\{1,2\}$ such that $(1)\sigma\alpha\tau=(2)\sigma\alpha\tau=c$. 
Moreover, if $x\alpha$ is odd or $n$ is even, we can guarantee the result with $c=1$. 
\end{lemma}
\begin{proof}
Let $b=x\alpha$.   
Without loss of generality, let us assume that $x$ is odd and $y$ is even.  
We begin by considering four cases involving $x$ and $y$. 
If $x=1$ and $y=2$, then take $\sigma=\iota_n$.  
If $x=1$ and $y\neq2$, then take $\sigma=(2,y)(3,5)\in\Gamma_n$ for $n\geqslant5$; 
if $n=4$, then $y=4$, and so take $\sigma=(4,3,2,1)\in\Gamma_4$. 
If $x\neq1$ and $y=2$, then take $\sigma=(1,x)(4,6)\in\Gamma_n$ for $n\geqslant6$; 
if $n=4$, then $x=3$, and so take $\sigma=(1,2,3,4)\in\Gamma_4$; 
if $n=5$, then $x=3$ or $x=5$, and so take $\sigma=(1,3,5)\in\Gamma_5$ in the first case and $\sigma=(1,5,3)\in\Gamma_5$ in the last case. 
Finally, if $x\neq1$ and $y\neq2$, then take $\sigma=(1,x)(2,y)\in\Gamma_n$. 
In any case, we get $(1)\sigma\alpha=(2)\sigma\alpha=b$. 

Now, we consider five cases that depend on $b$ and the parity of $n$. 
If $b=1$, then take $\tau=\iota_n$. 
If $b\neq1$ and $b$ is odd, then take $\tau=(1,b)(2,4)\in\Gamma_n$. 
If $b$ and $n$ are even, then take $\tau=(n,n-1,\ldots,b+1,2,b-1,\ldots,3,b,1)\in\Gamma_n$. 
In any of the cases already considered, we obtain $(1)\sigma\alpha\tau=(2)\sigma\alpha\tau=1$. 
Next, we consider the remaining cases. 
If $b=2$ and $n$ is odd, then take $\tau=\iota_n$. 
Finally, if $b$ is even, $b\neq2$ and $n$ is odd, then take $\tau=(1,3)(2,b)\in\Gamma_n$. 
In these cases, we obtain  $(1)\sigma\alpha\tau=(2)\sigma\alpha\tau=2$. 
\end{proof} 

Let $\alpha\in\T_n(\r=n-2)$. Then, $\alpha$ can only be one of the following two types: 
\begin{itemize}
\item ({\sc type a}) There exists $b\in\im(\alpha)$ such that $|b\alpha^{-1}|=3$;
\item ({\sc type b}) There exist $a,b\in\im(\alpha)$ such that $a\neq b$ and $|a\alpha^{-1}|=|b\alpha^{-1}|=2$. 
\end{itemize}
Observe that, if $\sigma,\tau\in\S_n$, then $\r(\sigma\alpha\tau)=n-2$ and $\alpha$ is of {\sc type a} if and only if $\sigma\alpha\tau$ is of {\sc type a}. 

\begin{proof}(of Theorem \ref{delta}) 
First, notice that, if $n=3$, then $\X=\emptyset$. 
As we have already noted, $\Delta_3=\Sigma_3$, so $\Delta_3=\Sigma_3\setdif\X$. 
Therefore, for the rest of the proof, we assume that $n\geqslant4$. 

Next, let us show that $\Delta_n\subseteq\Sigma_n\setdif\X$. Let $\alpha\in\Delta_n$. 
Then, $\alpha\in\Sigma_n$, given that $\Delta_n$ is generated by $\Sigma_n(\r\geqslant n-1)$. 
If $\r(\alpha)\geqslant n-1$, then $\alpha\not\in\X$, by the definition of $\X$, and so $\alpha\in\Sigma_n\setdif\X$. 
So, suppose that $\r(\alpha)\leqslant n-2$. Hence, there exist $\sigma\in\Gamma_n$ (possibly $\sigma=\iota_n$), 
$\beta\in\Sigma_n(\r=n-1)$ and $\gamma\in\Delta_n$ such that $\alpha=\sigma\beta\gamma$. 
Let $[\beta]=\{a,b\}$. Then, by Proposition \ref{n-1}, $a+b$ is odd. Let $x=a\sigma^{-1}$ and $y=b\sigma^{-1}$. 
As $\sigma^{-1}\in\PAP_n$ and $a+b$ is odd, then $x+y$ is also odd. 
Moreover, $x\alpha=x\sigma\beta\gamma=a\beta\gamma=b\beta\gamma=y\sigma\beta\gamma=y\alpha$ and so $\alpha\not\in\X$. 
Thus, we showed that $\Delta_n\subseteq\Sigma_n\setdif\X$. 

Now, we show the converse inclusion. Let $\alpha\in\Sigma_n\setdif\X$. 
If $\r(\alpha)\geqslant n-1$, then $\alpha\in\Delta_n$, as $\Delta_n$ is generated by $\Sigma_n(\r\geqslant n-1)$.
If $\r(\alpha)\leqslant n-3$, by Lemma \ref{delta1} and using simple inductive reasoning, $\alpha$ can be written as a product of transformations of  
$\Sigma_n\setdif\X$ of rank $n-2$. Thus, to complete the proof, it suffices to show that if $\r(\alpha)=n-2$, then $\alpha\in\Delta_n$.  
Moreover, in view of Lemma \ref{delta2}, we may assume $1\alpha=2\alpha=b\in\{1,2\}$. 
Let $b_1,b_2\in[n]$ be such that $b_1<b_2$ and $[n]\setdif\im(\alpha)=\{b_1<b_2\}$. 

First, suppose that $\alpha$ is of {\sc type a}. Let $u\in[n]$ be such that $u\geqslant3$ and $u\alpha=b$. 
Let us consider \textit{two} cases. 

\noindent{\sc case} 1. $b=1$ and $2\in\im(\alpha)$ [respectively, $b=2$ and $1\in\im(\alpha)$]. Observe that, in this case, $b_1\geqslant3$. 
Define $\alpha'_1,\alpha'_2,\alpha''_1,\alpha''_2\in\T_n$ by 
$$
i\alpha'_1=\left\{ \begin{array}{ll} 
1 & \mbox{if $i=1,2$}\\
2 & \mbox{if $i=u$}\\
b_2 & \mbox{if $i=v$}\\
i\alpha & \mbox{otherwise},  
\end{array}\right. 
j\alpha'_2=\left\{ \begin{array}{ll} 
1 & \mbox{if $j=1,2$}\\
2 & \mbox{if $j=b_2$}\\
j & \mbox{otherwise},  
\end{array}\right. 
i\alpha''_1=\left\{ \begin{array}{ll} 
1 & \mbox{if $i=1,2$}\\
2 & \mbox{if $i=u$}\\
b_1 & \mbox{if $i=v$}\\
i\alpha & \mbox{otherwise} 
\end{array}\right. 
~\mbox{and}~
j\alpha''_2=\left\{ \begin{array}{ll} 
1 & \mbox{if $j=1,2$}\\
2 & \mbox{if $j=b_1$}\\
b_1 & \mbox{if $j=b_2$}\\
j & \mbox{otherwise}, 
\end{array}\right. 
$$ 
with $v\in[n]$ such that $v\alpha=2$ [respectively, 
$$
i\alpha'_1=\left\{ \begin{array}{ll} 
2 & \mbox{if $i=1,2$}\\
1 & \mbox{if $i=u$}\\
b_2 & \mbox{if $i=v$}\\
i\alpha & \mbox{otherwise},  
\end{array}\right. 
j\alpha'_2=\left\{ \begin{array}{ll} 
2 & \mbox{if $j=1,2$}\\
1 & \mbox{if $j=b_2$}\\
j & \mbox{otherwise},  
\end{array}\right. 
i\alpha''_1=\left\{ \begin{array}{ll} 
2 & \mbox{if $i=1,2$}\\
1 & \mbox{if $i=u$}\\
b_1 & \mbox{if $i=v$}\\
i\alpha & \mbox{otherwise} 
\end{array}\right. 
~\mbox{and}~
j\alpha''_2=\left\{ \begin{array}{ll} 
2 & \mbox{if $j=1,2$}\\
1 & \mbox{if $j=b_1$}\\
b_1 & \mbox{if $j=b_2$}\\
j & \mbox{otherwise}, 
\end{array}\right. 
$$ 
with $v\in[n]$ such that $v\alpha=1$]. 
Then, $\alpha=\alpha'_1\alpha'_2=\alpha''_1\alpha''_2$, $\im(\alpha'_1)=\im(\alpha)\cup\{b_2\}$ and 
$\im(\alpha'_2)=\im(\alpha''_1)=\im(\alpha''_2)=\im(\alpha)\cup\{b_1\}$, and so 
$\r(\alpha'_1)=\r(\alpha'_2)=\r(\alpha''_1)=\r(\alpha''_2)=n-1$. 
Moreover, ${\alpha_2'} _{|X_1}$ has $b_2-3$ [respectively, $b_2-2$] inversions and ${\alpha''_2}_{|X_1}$ has $b_2-4$ [respectively, $b_2-3$] inversions. 
Let us take 
$$
\alpha_2=\left\{ \begin{array}{ll} 
\alpha'_2 & \mbox{if $b_2$ is odd [respectively, even]}\\
\alpha''_2 & \mbox{if $b_2$ is even [respectively, odd]}. 
\end{array}\right. 
$$
Hence, ${\alpha_2}_{|X_1}$ is even and so $\alpha_2\in\Delta_n$. Next, observe that $(1,2)\alpha'_2=\alpha'_2$ and $(1,2)\alpha''_2=\alpha''_2$. 
On the other hand, 
$\inv((\alpha'_1(1,2))_{|X_1})=\inv({\alpha'_1}_{|X_1})\cup\{\{2,u\}\}$ and $\inv((\alpha''_1(1,2))_{|X_1})=\inv({\alpha''_1}_{|X_1})\cup\{\{2,u\}\}$ 
[respectively, 
$\inv({\alpha'_1}_{|X_1})=\inv((\alpha'_1(1,2))_{|X_1})\cup\{\{2,u\}\}$ and $\inv({\alpha''_1}_{|X_1})=\inv((\alpha''_1(1,2))_{|X_1})\cup\{\{2,u\}\}$], 
whence ${\alpha'_1}_{|X_1}$ is even if and only if $(\alpha'_1(1,2))_{|X_1}$ is odd, and 
${\alpha''_1}_{|X_1}$ is even if and only if $(\alpha''_1(1,2))_{|X_1}$ is odd. 
Let us take 
$$
\alpha_1=\left\{ \begin{array}{ll} 
\alpha'_1 & \mbox{if ${\alpha'_1}_{|X_1}$ is even and $b_2$ is odd [respectively, even]}\\
\alpha'_1(1,2) & \mbox{if ${\alpha'_1}_{|X_1}$ is odd and $b_2$ is odd [respectively, even]}\\ 
\alpha''_1 & \mbox{if ${\alpha''_1}_{|X_1}$ is even and $b_2$ is even [respectively, odd]}\\
\alpha''_1(1,2) & \mbox{if ${\alpha'_1}_{|X_1}$ is odd and $b_2$ is even [respectively, odd].}
\end{array}\right. 
$$
Therefore, ${\alpha_1}_{|X_1}$ is even and so $\alpha_1\in\Delta_n$. Moreover, 
$$
\begin{array}{rcl}
\alpha_1\alpha_2 & = & \left\{ \begin{array}{ll} 
\alpha'_1\alpha'_2 & \mbox{if ${\alpha'_1}_{|X_1}$ is even and $b_2$ is odd [respectively, even]}\\
\alpha'_1(1,2)\alpha'_2 & \mbox{if ${\alpha'_1}_{|X_1}$ is odd and $b_2$ is odd [respectively, even]}\\ 
\alpha''_1\alpha''_2 & \mbox{if ${\alpha''_1}_{|X_1}$ is even and $b_2$ is even [respectively, odd]}\\
\alpha''_1(1,2)\alpha''_2 & \mbox{if ${\alpha'_1}_{|X_1}$ is odd and $b_2$ is even [respectively, odd]}
\end{array}\right. \\ [9mm]
 & = &  \left\{ \begin{array}{ll} 
\alpha'_1\alpha'_2 & \mbox{if $b_2$ is odd [respectively, even]}\\
\alpha''_1\alpha''_2 & \mbox{if $b_2$ is even [respectively, odd]}
\end{array}\right. \\ [3mm]
& = & \alpha. 
\end{array}
$$
Thus, $\alpha\in\Delta_n$. 

\noindent{\sc case} 2. $b=1$ and $2\not\in\im(\alpha)$ [respectively, $b=2$ and $1\not\in\im(\alpha)$]. In this case, we have $b_1=2$ [respectively, $b_1=1$]. 
Define $\alpha'_1,\alpha_2\in\T_n$ by 
$$
i\alpha'_1=\left\{ \begin{array}{ll} 
\mbox{$1$ [respectively, $2$]} & \mbox{if $i=1,2$}\\
\mbox{$2$ [respectively, $1$]}  & \mbox{if $i=u$}\\
i\alpha & \mbox{otherwise}  
\end{array}\right. 
~\mbox{and}~ 
j\alpha_2=\left\{ \begin{array}{ll} 
\mbox{$1$ [respectively, $2$]}  & \mbox{if $j=1,2$}\\
j & \mbox{otherwise.}  
\end{array}\right. 
$$ 
Then, $\alpha=\alpha'_1\alpha_2$, $\im(\alpha'_1)=\im(\alpha)\cup\{2\}$ and $\im(\alpha_2)=[n]\setdif\{2\}$ 
[respectively, $\im(\alpha'_1)=\im(\alpha)\cup\{1\}$ and $\im(\alpha_2)=[n]\setdif\{1\}$], 
whence $\r(\alpha'_1)=\r(\alpha_2)=n-1$. Clearly, 
$\alpha_2\in\Delta_n$ (note that $\alpha_2$ is an \textit{order-preserving} transformation) and 
$(1,2)\alpha_2=\alpha_2$. 
On the other hand, 
$\inv((\alpha'_1(1,2))_{|X_1})=\inv({\alpha'_1}_{|X_1})\cup\{\{2,u\}\}$ 
[respectively, 
$\inv({\alpha'_1}_{|X_1})=\inv((\alpha'_1(1,2))_{|X_1})\cup\{\{2,u\}\}$] 
and so ${\alpha'_1}_{|X_1}$ is even if and only if $(\alpha'_1(1,2))_{|X_1}$ is odd. 
Let us take 
$$
\alpha_1=\left\{ \begin{array}{ll} 
\alpha'_1 & \mbox{if ${\alpha'_1}_{|X_1}$ is even}\\
\alpha'_1(1,2) & \mbox{if ${\alpha'_1}_{|X_1}$ is odd.}
\end{array}\right. 
$$
Therefore, ${\alpha_1}_{|X_1}$ is even and so $\alpha_1\in\Delta_n$. Moreover, 
$$
\alpha_1\alpha_2=\left\{ \begin{array}{ll} 
\alpha'_1\alpha_2 & \mbox{if ${\alpha'_1}_{|X_1}$ is even}\\
\alpha'_1(1,2)\alpha_2 & \mbox{if ${\alpha'_1}_{|X_1}$ is odd.}
\end{array}\right. 
=\alpha'_1\alpha_2=\alpha. 
$$
Thus, $\alpha\in\Delta_n$. 

Secondly, suppose that $\alpha$ is of {\sc type b}. Let $u,v, a\in[n]$ be such that $2<u<v$ and 
$u\alpha=v\alpha=a$. Let us consider two cases again.  

\noindent{\sc case} 1. $a$, $b_1$ and $b_2$ do not all have the same parity.  Define $\alpha'_1,\alpha_2\in\T_n$ by 
$$
i\alpha'_1=\left\{ \begin{array}{ll} 
b & \mbox{if $i=1,2$}\\
b_1 & \mbox{if $i=u$}\\
a & \mbox{if $i=v$}\\
i\alpha & \mbox{otherwise}
\end{array}\right. 
~\mbox{and}~ 
j\alpha_2=\left\{ \begin{array}{ll} 
a & \mbox{if $j=b_1,b_2$}\\
j & \mbox{otherwise.}
\end{array}\right. 
$$
Then, $\alpha=\alpha'_1\alpha_2$, $\im(\alpha'_1)=\im(\alpha)\cup\{b_1\}$ and $\im(\alpha_2)=\im(\alpha)$, 
whence $\r(\alpha'_1)=n-1$ and $\r(\alpha_2)=n-2$. Moreover, it is clear that $\alpha_2\in\Sigma_n(\r=n-2)\setdif\X$ 
is of {\sc type a} (notice that, we also have $a\alpha_2=a$) 
and so, based on what we have shown above, $\alpha_2\in\Delta_n$. 

\noindent{\sc case} 2. $a$, $b_1$ and $b_2$ all have the same parity.  In this case, we must have $n\geqslant 5$.  
If $a+c$ is even for any $c\in\im(\alpha)\setdif\{a,b\}$, then $[n]$ has at least $n-1$ elements with the same parity, which holds only for $n \leqslant 3$. 
Hence, there exists $c\in\im(\alpha)\setdif\{a,b\}$ such that $a+c$ is odd. Let $w\in[n]$ be such that $w\alpha=c$. Define 
$\alpha'_1,\alpha_2\in\T_n$ by 
$$
i\alpha'_1=\left\{ \begin{array}{ll} 
b & \mbox{if $i=1,2$}\\
b_1 & \mbox{if $i=u$}\\
a & \mbox{if $i=v$}\\
b_2 & \mbox{if $i=w$}\\
i\alpha & \mbox{otherwise}
\end{array}\right. 
~\mbox{and}~ 
j\alpha_2=\left\{ \begin{array}{ll} 
a & \mbox{if $j=c,b_1$}\\
c & \mbox{if $j=b_2$}\\
j & \mbox{otherwise.}
\end{array}\right. 
$$
Then, $\alpha=\alpha'_1\alpha_2$, $\im(\alpha'_1)=[n]\setdif\{c\}$ and $\im(\alpha_2)=\im(\alpha)$, 
whence $\r(\alpha'_1)=n-1$ and $\r(\alpha_2)=n-2$. Moreover, it is clear that $\alpha_2\in\Sigma_n(\r=n-2)\setdif\X$ 
is of {\sc type a} (notice that, in this case, we also have $a\alpha_2=a$) 
and so, as above, we get $\alpha_2\in\Delta_n$. 

Now, let $\sigma=(a,b_1)\in\S_n$. In both cases, we have $\sigma\alpha_2=\alpha_2$ and so $(\alpha'_1\sigma)\alpha_2=\alpha'_1\alpha_2=\alpha$. 
On the other hand, 
$$
(\alpha'_1\sigma)_{|X_1}={\alpha'_1}_{|X_1}\sigma_{|X_1\alpha'_1}={\alpha'_1}_{|X_1}\sigma_{|[n]\setdif\{d\}},
$$
with $d=b_2$ in {\sc case 1} and $d=c$ in {\sc case 2}.  
Since $d\not\in\{a,b_1\}$, then $\sigma_{|[n]\setdif\{d\}}=(a,b_1)$ is an odd permutation of $[n]\setdif\{d\}$. 
Hence, by \ref{lema2}, if ${\alpha'_1}_{|X_1}$ is odd, then $(\alpha'_1\sigma)_{|X_1}$ is even. 
Take 
$$
\alpha_1=\left\{ \begin{array}{ll} 
\alpha'_1 & \mbox{if ${\alpha'_1}_{|X_1}$ is even}\\
\alpha'_1\sigma & \mbox{if ${\alpha'_1}_{|X_1}$ is odd.}
\end{array}\right. 
$$
Therefore, ${\alpha_1}_{|X_1}$ is even and so $\alpha_1\in\Delta_n$. Moreover, 
$$
\alpha_1\alpha_2=\left\{ \begin{array}{ll} 
\alpha'_1\alpha_2 & \mbox{if ${\alpha'_1}_{|X_1}$ is even}\\
(\alpha'_1\sigma)\alpha_2 & \mbox{if ${\alpha'_1}_{|X_1}$ is odd.}
\end{array}\right. 
=\alpha'_1\alpha_2=\alpha. 
$$
Thus, $\alpha\in\Delta_n$. 
The proof is now complete. 
\end{proof} 

For any integers $a$ and $b$, denote by $S(a,b)$ the Stirling number of the second kind of $a$ and $b$. 
Observe that, for $a>0$, if $b\leqslant0$ or $b>a$, then $S(a,b)=0$. 

Let $1\leqslant r\leqslant n-2$. By Theorem \ref{delta}, if $n$ is odd, for each $1\leqslant s\leqslant\frac{n+1}{2}$, 
we have $S(\frac{n+1}{2},s)S(\frac{n-1}{2},r-s)$ different possible kernels for an element of $\Sigma_n\setdif\Delta_n$ with rank $r$; 
 if $n$ is even, for each $1\leqslant s\leqslant\frac{n}{2}$, 
we have $S(\frac{n}{2},s)S(\frac{n}{2},r-s)$ different possible kernels for an element of $\Sigma_n\setdif\Delta_n$ with rank $r$. 
Therefore, we get 
\begin{equation}\label{deltasize} 
|\Sigma_n\setdif\Delta_n| = \left\{
\begin{array}{ll}
\sum_{r=1}^{n-2}\binom{n}{r}r!\sum_{s=1}^{\frac{n+1}{2}} S(\frac{n+1}{2},s)S(\frac{n-1}{2},r-s)& \mbox{if $n$ is odd}\\ [1mm] 
\sum_{r=1}^{n-2}\binom{n}{r}r!\sum_{s=1}^{\frac{n}{2}} S(\frac{n}{2},s)S(\frac{n}{2},r-s) & \mbox{if $n$ is even}, 
\end{array}
\right.
\end{equation} 
for $n\geqslant3$. 
Obviously,  we get an explicit formula for $|\Delta_n| (=|\Sigma_n|-|\Sigma_n\setdif\Delta_n|)$ 
from (\ref{sigmasize}) and (\ref{deltasize}). 

$$\begin{array}{c|c|c|c|c|c}  
n & |\Gamma_n| & |\S_n|  & |\Delta_n| & |\Sigma_n| & |\T_n|\\ \hline 
1 & 1&  1& 1& 1& 1\\ 
2 & 2&  2& 4& 4 & 4 \\ 
3 & 1&  6& 10& 10 &27\\ 
4 & 4&  24& 128& 140 & 256\\ 
5 & 6&  120& 1911& 2171 & 3125\\ 
6 & 36&  720& 33702& 38412 & 46656\\ 
7 & 72&  5040& 651793& 742975 & 823543\\ 
8 & 576&  40320& 14237912& 15931072 & 16777216\\ 
9 &1440&  362880& 342062865 & 377624169 & 387420489\\ 
10 & 14400&  3628800& 9120890710& 9878449600 & 10000000000\\ 
\end{array}$$

\section{Generators and rank of $\Delta_n$ and $\Sigma_n$}\label{rankdeltasigma}

In the final section of this paper, we present minimal generating sets and determine the ranks of the monoids $\Delta_n$ and $\Sigma_n$. 
We also show that these monoids are both regular. 

Let 
$$
\varepsilon_n=\transf{1&2&3&\cdots&n\\1&1&3&\cdots&n}\quad\mbox{and}\quad\varepsilon'_n=\transf{1&2&3&\cdots&n\\2&2&3&\cdots&n}
$$
for $n\geqslant4$. Clearly, $\varepsilon_n,\varepsilon'_n\in\Sigma_n(\r=n-1)$. 

Let us consider the following subsets of $\Delta_n$: 
$$
\B_n=\{\beta\in\Sigma_n(\r=n-1)\mid \mbox{$1\beta=2\beta=1$ and $2\not\in\im(\beta)$}\}
$$
and 
$$
\B'_n=\{\beta\in\Sigma_n(\r=n-1)\mid \mbox{$1\beta=2\beta=2$ and $1\not\in\im(\beta)$}\}. 
$$
For any $\beta\in \B_n\cup \B'_n$, as $\beta_{|X_1}$ is even, $\beta_{|\{3,\ldots,n\}}$ is an even permutation of $\{3,\ldots,n\}$. 
Therefore, $\B_n$ and $\B'_n$ are subgroups of $\Delta_n$ with identities $\varepsilon_n$ and $\varepsilon'_n$, respectively, 
isomorphic to $\A(\{3,\ldots,n\})$. The mapping $\B_n\longrightarrow\A(\{3,\ldots,n\})$ (respectively, $\B'_n\longrightarrow\A(\{3,\ldots,n\})$), 
$\beta\mapsto \beta_{|\{3,\ldots,n\}}$ is clearly an isomorphism. 

Now, recall that, for $m\geqslant3$, it is well known that 
$$
A_m=\left\{\begin{array}{ll}
\langle (1,2,3),(1,2,\ldots,m)\rangle & \mbox{if $m$ is odd}\\
\langle (1,2,3),(2,3,\ldots,m)\rangle & \mbox{if $m$ is even}, 
\end{array}\right.
$$
so for $n\geqslant 5$ 
$$
\A(\{3,\ldots,n\})=\left\{\begin{array}{ll}
\langle (3,4,5),(3,4,\ldots,n)\rangle & \mbox{if $n$ is odd}\\
\langle (3,4,5),(4,5,\ldots,n)\rangle & \mbox{if $n$ is even}. 
\end{array}\right.
$$
Let 
$$
\beta_1=\transf{1&2  & 3&4&5 & 6&\cdots&n\\1&1  & 4&5&3 & 6&\cdots&n}, 
\quad
\beta_2=\left\{\begin{array}{ll}
\transf{1&2  & 3&\cdots&n-1&n\\1&1  & 4&\cdots&n&3} &\mbox{if $n$ is odd}\\ [3mm] 
\transf{1&2  & 3&4&\cdots&n-1&n\\1&1  &3& 5&\cdots&n&4} &\mbox{if $n$ is even}
\end{array}\right. 
$$
$$
\beta'_1=\transf{1&2  & 3&4&5 & 6&\cdots&n\\2&2  & 4&5&3 & 6&\cdots&n} 
\quad\mbox{and}\quad 
\beta'_2=\left\{\begin{array}{ll}
\transf{1&2  & 3&\cdots&n-1&n\\2&2  & 4&\cdots&n&3} &\mbox{if $n$ is odd}\\ [3mm] 
\transf{1&2  & 3&4&\cdots&n-1&n\\2&2  &3& 5&\cdots&n&4} &\mbox{if $n$ is even} 
\end{array}\right. 
$$
for $n\geqslant 5$. Then, given the isomorphisms mentioned above, we have $\B_n=\langle\beta_1,\beta_2\rangle$ and $\B'_n=\langle\beta'_1,\beta'_2\rangle$. 

\begin{lemma}\label{bes}
For $n\geqslant5$, $\B_n\cup \B'_n\subseteq\langle\Gamma_n,\varepsilon_n,\varepsilon'_n\rangle$.  
\end{lemma}
\begin{proof}
We begin by noticing that $\beta_1=\varepsilon_n(1,5,3)\varepsilon'_n(2,4)(3,5)\varepsilon_n(1,5,3)\in\langle\Gamma_n,\varepsilon_n,\varepsilon'_n\rangle$. 
Next, we consider the cases where $n$ is odd and $n$ is even separately. 

First, let us assume that $n$ is odd. For each odd $i\in\{5,7,\ldots,n\}$, define 
$$
\beta_{2,i}=\transf{1&2  & 3 & 4&\cdots&i-2 & i-1&i & i+1&\cdots& n\\1&1  & i-1 & 4&\cdots&i-2 & i&3 & i+1&\cdots& n}
$$
(observe that ${\beta_{2,i}}_{|X_1}$ has $2i-8$ inversions, whence $\beta_{2,i}\in\Sigma_n(\r=n-1)$).  
Then, we have $\beta_2=\beta_{2,5}\beta_{2,7}\cdots\beta_{2,n}$ and 
$$
\beta_{2,i}= \varepsilon_n(1,i,3)\varepsilon'_n(2,i-1)(3,i)\varepsilon_n(1,i,3) \in \langle\Gamma_n,\varepsilon_n,\varepsilon'_n\rangle, 
$$
for $i\in\{5,7,\ldots,n\}$, and so $\beta_2\in\langle\Gamma_n,\varepsilon_n,\varepsilon'_n\rangle$. 

Next, we proceed in a similar manner if $n$ is even. For each even $i\in\{6,8,\ldots,n\}$, define 
$$
\beta_{2,i}=\transf{1&2 &3 & 4 & 5&\cdots&i-2 & i-1&i & i+1&\cdots& n\\1&1 &3 & i-1 & 5&\cdots&i-2 & i&4 & i+1&\cdots& n}
$$
(${\beta_{2,i}}_{|X_1}$ has $2i-10$ inversions, so $\beta_{2,i}\in\Sigma_n(\r=n-1)$).  
Since $\beta_2=\beta_{2,6}\beta_{2,8}\cdots\beta_{2,n}$ and 
$$
\beta_{2,i}= \varepsilon_n(1,3,i-1)\varepsilon'_n(2,i,4)\varepsilon_n(1,i-1,3) \in \langle\Gamma_n,\varepsilon_n,\varepsilon'_n\rangle, 
$$
for $i\in\{6,8,\ldots,n\}$, we get $\beta_2\in\langle\Gamma_n,\varepsilon_n,\varepsilon'_n\rangle$. 

Thus, in both cases, $\beta_1,\beta_2\in\langle\Gamma_n,\varepsilon_n,\varepsilon'_n\rangle$, whence $\B_n\subseteq\langle\Gamma_n,\varepsilon_n,\varepsilon'_n\rangle$. 
On the other hand, $\beta'_1=\beta_1\varepsilon'_n$ and $\beta'_2=\beta_2\varepsilon'_n$, so we also have 
$\B'_n\subseteq\langle\Gamma_n,\varepsilon_n,\varepsilon'_n\rangle$, as required. 
\end{proof} 

Observe that, for an even $n\geqslant6$, we also have 
$
\beta_1 = (\varepsilon_n\sigma_n^{-1})^2\varepsilon_n(1,5,7,\ldots,n-1)(4,6,\ldots,n)
$
and 
$$
\beta_2=(\varepsilon_n(1,n,3,4,\ldots,n-1,2))^{n-4}\varepsilon_n(1,3,\ldots,n-1)^2(4,6,\ldots,n)^2.  
$$
Since $\sigma_n, (1,5,7,\ldots,n-1)(4,6,\ldots,n), (1,n,3,4,\ldots,n-1,2), (1,3,\ldots,n-1)^2, (4,6,\ldots,n)^2\in \Gamma_n$, 
we get $\beta_1,\beta_2\in\langle\Gamma_n,\varepsilon_n\rangle$ and so $\B_n\subseteq\langle\Gamma_n,\varepsilon_n\rangle$. 
Furthermore, for an even $n\geqslant4$, we have $\varepsilon'_n=(\varepsilon_n\sigma_n^{-1})^{n-1}$ and 
$\varepsilon_n=(\varepsilon'_n\sigma_n)^{n-1}$, whence $\langle \sigma_n,\varepsilon_n\rangle=\langle \sigma_n,\varepsilon'_n\rangle$ 
and, in particular, $\langle\Gamma_n,\varepsilon_n,\varepsilon'_n\rangle=\langle \Gamma_n,\varepsilon_n\rangle=\langle \Gamma_n,\varepsilon'_n\rangle$. 

\smallskip 

Next, for each $3\leqslant i\leqslant n$, let 
$$
\lambda_i = \transf{1&2 & 3&4&\cdots &i &i+1& \cdots & n \\ 1&1 & 2&3&\cdots &i-1 &i+1& \cdots & n}. 
$$

\begin{lemma}\label{lambdai}
For $n\geqslant5$, $\lambda_3,\lambda_4,\ldots,\lambda_n\in\langle\Gamma_n,\varepsilon_n,\varepsilon'_n\rangle$.  
\end{lemma}
\begin{proof}
We have $\lambda_3=\varepsilon_n(1,5,3)\varepsilon'_n(1,3,5)$ and 
$$
\lambda_i = \left\{\begin{array}{ll}
\lambda_{i-1}(1,i-1)(2,i)\varepsilon_n(1,i-1)(2,i) & \mbox{if $i$ is even}\\[1mm] 
\lambda_{i-1}(1,i)(2,i-1)\varepsilon'_n(1,i)(2,i-1) & \mbox{if $i$ is odd}, 
\end{array}\right. 
$$
for $4\leqslant i\leqslant n$. 
Then, the lemma follows by induction on $i$. 
\end{proof} 

Observe that, for an even $n\geqslant6$, we also explicitly have $\lambda_i=\varepsilon_n(\sigma_n^{-1}\varepsilon_n)^{i-2}\sigma_n^{i-2}\in\langle\Gamma_n,\varepsilon_n\rangle$ 
for $3\leqslant i\leqslant n$.

\begin{theorem}\label{deltarank}
For $n\geqslant4$, $\Delta_n=\langle \Gamma_n, \varepsilon_n\rangle=\langle \Gamma_n, \varepsilon'_n\rangle$, if $n$ is even, 
and $\Delta_n=\langle \Gamma_n, \varepsilon_n,\varepsilon'_n\rangle$, if $n$ is odd. 
Moreover, the rank of $\Delta_n$ is 
$$
\left\{
\begin{array}{ll}
2 & \mbox{if $n=4$}\\
4 & \mbox{if $n$ is odd, $n\geqslant5$ and $n\neq7$}\\
3 & \mbox{if $n$ is even and $n\geqslant6$}\\
5 & \mbox{if $n=7$}. 
\end{array}
\right.
$$
\end{theorem}
\begin{proof}
Let $\alpha\in\Sigma_n(\r=n-1)$. Then, by Lemma \ref{delta2}, there exist $\sigma,\tau\in\Gamma_n$ such that 
$(1)\sigma\alpha\tau=(2)\sigma\alpha\tau=c\in\{1,2\}$. Let $\beta=\sigma\alpha\tau$. 
Notice that $\beta\in\Sigma_n(\r=n-1)$. 
Our goal is to prove that $\beta\in \langle\Gamma_n,\varepsilon_n,\varepsilon'_n\rangle$. 

First of all, we consider the case $n=4$ separately.  Note that, since $n$ is even, we can assume that $c=1$.  
Since $\beta_{|X_1}$ is even, it follows that 
$$
\beta\in\left\{ \transf{1&2&3&4\\1&1&2&3}, \transf{1&2&3&4\\1&1&2&4}, \varepsilon_4=\transf{1&2&3&4\\1&1&3&4}\right\}. 
$$
Then, as  
$$
\transf{1&2&3&4\\1&1&2&4}=\varepsilon_4\sigma_4^{-1}\varepsilon_4\sigma_4 \quad\mbox{and}\quad 
\transf{1&2&3&4\\1&1&2&3}=(\varepsilon_4\sigma_4^{-1})^2\varepsilon_4\sigma_4^2, 
$$
we get $\beta\in\langle \Gamma_4,\varepsilon_4\rangle$. 

Now, suppose that $n\geqslant5$. Let us consider the two possible cases for $c$. 

\noindent{\sc case} $c=1$. If $2\not\in\im(\beta)$, then $\beta\in \B_n$, and so $\beta\in \langle\Gamma_n,\varepsilon_n,\varepsilon'_n\rangle$, 
by Lemma \ref{bes}. So, let us suppose that $2\in\im(\beta)$. Then, $\im(\beta)=[n]\setdif\{i\}$ for some $3\leqslant i\leqslant n$. 
Let 
$$
\lambda'_i=\transf{1& 2&3&\cdots&i-1& i&\cdots&n\\1& 3&4&\cdots&i& i&\cdots&n}. 
$$
Clearly, $\lambda'_i\in\Sigma_n(\r=n-1)$, whence $\beta\lambda'_i\in\Delta_n$, and $\beta\lambda'_i\lambda_i=\beta$. 
Moreover, $(1)\beta\lambda'_i=(2)\beta\lambda'_i=1$ and $\im(\beta\lambda'_i)=[n]\setdif\{2\}$, 
whence $\beta\lambda'_i\in \B_n$ and so $\beta\lambda'_i\in \langle\Gamma_n,\varepsilon_n,\varepsilon'_n\rangle$, 
by Lemma \ref{bes}. 
Since $\lambda_i\in\langle\Gamma_n,\varepsilon_n,\varepsilon'_n\rangle$, 
by Lemma \ref{lambdai}, we get $\beta= (\beta\lambda'_i)\lambda_i\in\langle\Gamma_n,\varepsilon_n,\varepsilon'_n\rangle$. 

\noindent{\sc case} $c=2$. If $1\not\in\im(\beta)$, then $\beta\in \B'_n$, and so $\beta\in \langle\Gamma_n,\varepsilon_n,\varepsilon'_n\rangle$, 
by Lemma \ref{bes}. So, suppose that $1\in\im(\beta)$. Then, in this case, we also have $\im(\beta)=[n]\setdif\{i\}$ for some $3\leqslant i\leqslant n$. 

First, suppose that $i$ is odd and let 
$$
\xi_i=\transf{1& 2&3&\cdots&i-1& i& i+1&\cdots&n\\1& 1&3&\cdots&i-1&2& i+1&\cdots&n} 
\quad\mbox{and}\quad 
\xi'_i=\transf{1& 2&3&\cdots&i-1& i& i+1&\cdots&n\\2& i&3&\cdots&i-1&i& i+1&\cdots&n}. 
$$ 
Then, we have $\beta = \beta\xi'_i\xi_i$. 
Since ${\xi_i}_{|X_1}$ has $i-3$ inversions, it is clear that $\xi_i\in\Sigma_n(\r=n-1)$, 
whence $\xi_i\in \langle\Gamma_n,\varepsilon_n,\varepsilon'_n\rangle$, by {\sc case} $c=1$. 
On the other hand, as ${\xi'_i}_{|X_2}$ has no inversions, $2\xi'_i=i\xi'_i=i$, $\im(\xi'_i)=[n]\setdif\{1\}$ and $2+i$ is odd, 
we get  $\xi'_i\in\Sigma_n(\r=n-1)$. 
Moreover, we also have $\im(\beta\xi'_i)=[n]\setdif\{1\}$, which allow us to conclude that $\beta\xi'_i\in\Sigma_n(\r=n-1)$ as well. 
In addition, since $(1)\beta\xi'_i=(2)\beta\xi'_i=i$ is odd, Lemma \ref{delta2} guarantees that 
there exist $\sigma_0,\tau_0\in\Gamma_n$ such that 
$(1)\sigma_0(\beta\xi'_i)\tau_0=(2)\sigma_0(\beta\xi'_i)\tau_0=1$. 
Hence, by {\sc case} $c=1$, we get $\sigma_0\beta\xi'_i\tau_0\in  \langle\Gamma_n,\varepsilon_n,\varepsilon'_n\rangle$ 
and so $\beta = \beta\xi'_i\xi_i = \sigma_0^{-1}(\sigma_0\beta\xi'_i\tau_0)\tau_0^{-1}\xi_i\in  \langle\Gamma_n,\varepsilon_n,\varepsilon'_n\rangle$. 

Next, we suppose that $i$ is even.  In this case, we take 
$$
\nu_i=\transf{1& 2&3&\cdots&i-1& i& i+1&\cdots&n\\2& 2&3&\cdots&i-1&1& i+1&\cdots&n} 
\quad\mbox{and}\quad 
\nu'_i=\transf{1& 2&3&\cdots&i-1& i& i+1&\cdots&n\\i& 1&3&\cdots&i-1&i& i+1&\cdots&n} 
$$ 
and also get $\beta = \beta\nu'_i\nu_i$. Since ${\nu_i}_{|X_1}$ has $i-2$ inversions and ${\nu'_i}_{|X_1}$ has no inversions, 
we have $\nu_i,\nu'_i\in\Sigma_n(\r=n-1)$. It follows that $\beta\nu'_i\in\Sigma_n(\r=n-1)$. 
Moreover, as $(1)\beta\nu'_i=(2)\beta\nu'_i=1$ and $\im(\beta\nu'_i)=[n]\setdif\{2\}$, 
we obtain $\beta\nu'_i\in \B_n$ and so 
$\beta\nu'_i\in \langle\Gamma_n,\varepsilon_n,\varepsilon'_n\rangle$, by Lemma \ref{bes}. 
On the other hand, we have 
$\nu_i=\varepsilon'_n(2,i)(3,5)\varepsilon_n(2,i)(3,5)\in \langle\Gamma_n,\varepsilon_n,\varepsilon'_n\rangle$,
so $\beta=(\beta\nu'_i)\nu_i\in \langle\Gamma_n,\varepsilon_n,\varepsilon'_n\rangle$. 

Thus, we have shown in all cases that $\beta\in \langle\Gamma_n,\varepsilon_n,\varepsilon'_n\rangle$. 
Consequently, $\alpha=\sigma^{-1}\beta\tau^{-1}\in \langle\Gamma_n,\varepsilon_n,\varepsilon'_n\rangle$. 
Therefore, $\Delta_n=\langle\Gamma_n,\varepsilon_n,\varepsilon'_n\rangle$. 
If $n$ is even,  by what we have already observed, we also get  
$\Delta_n=\langle\Gamma_n,\varepsilon_n\rangle=\langle\Gamma_n,\varepsilon'_n\rangle$. 

All that remains now is to discuss the rank of $\Delta_n$. 
Clearly, for an even $n$, we have $\rank(\Delta_n) = \rank(\Gamma_n) + 1$. 
On the other hand, for an odd $n$, if we show that we need at least two generators belonging to $\Delta_n \setdif \Gamma_n$, 
the stated result follows from Theorem \ref{gamman}. 

Suppose that $n$ is odd. Recall that, in this case, for all $i\in[n]$ and $\theta\in\Gamma_n$, $i$ and $i\theta$ have the same parity. 
Let $C$ be a set of generators of $\Delta_n$. Let $c\in\{1,2\}$ and 
$
\varepsilon=\smallt{1&2&3&\cdots&n\\c&c&3&\cdots&n}. 
$
Let $\alpha_1,\ldots,\alpha_k\in C$ ($k\in\N$) be such that $\varepsilon=\alpha_1\cdots\alpha_k$. 
Then, $\alpha_1,\ldots,\alpha_k\in\Sigma_n(\r\geqslant n-1)$ and there exists $1\leqslant i\leqslant k$ such that 
$\alpha_i\in\Sigma_n(\r=n-1)$ and $\alpha_{i+1},\ldots,\alpha_k\in\Gamma_n$. Let $\theta=\alpha_{i+1}\cdots\alpha_k\in\Gamma_n$ 
($\theta$ can be the identity, e.g. if $i=k$). 
Let $[\alpha_i]=\{x,y\}$ and let $x_0,y_0\in[n]$ be such that $(x_0)\alpha_1\cdots\alpha_{i-1}=x$ and $(y_0)\alpha_1\cdots\alpha_{i-1}=y$ 
(if $i=1$, then $\alpha_1\cdots\alpha_{i-1}=\iota_n$, and so $x_0=x$ and $y_0=y$). 
Then, 
$$
x_0\varepsilon = (x_0)\alpha_1\cdots\alpha_{i-1}\alpha_i\theta=(x)\alpha_i\theta=
(y)\alpha_i\theta=(y_0)\alpha_1\cdots\alpha_{i-1}\alpha_i\theta=y_0\varepsilon,
$$
whence $\{x_0,y_0\}=\{1,2\}$ and so $x\alpha_i\theta=c$. Thus, $x\alpha_i=y\alpha_i$ and $c$ have the same parity.  
It follows that $C$ must have at least two elements belonging to $\Sigma_n(\r=n-1)$, as required. 
\end{proof}

Some facts established in the proof of Theorem \ref{deltarank} help us prove the following result. 

\begin{proposition}
The monoids $\Delta_n$ and $\Sigma_n$ are regular. 
\end{proposition}
\begin{proof}
Let us begin by noting that, regarding the proof of the regularity of $\Sigma_n$, since it is well known that $\T_n$ is regular, 
it suffices to show that the elements of $\Sigma_n(\r=n-1)$ are regular in $\Delta_n$ (and thus in $\Sigma_n$). 
Note also that every group-element is regular; in particular, the elements of $\Gamma_n$ are regular (in $\Delta_n$ and in $\Sigma_n$). 
In addition, observe that multiplying an element of a monoid by a unit preserves regularity. 

Let $\alpha\in\Sigma_n(\r=n-1)$. Then, by Lemma \ref{delta2}, there exist $\sigma,\tau\in\Gamma_n$ such that 
$(1)\sigma\alpha\tau=(2)\sigma\alpha\tau=c\in\{1,2\}$. Let $\beta=\sigma\alpha\tau$. 
So, we have $\beta\in\Sigma_n(\r=n-1)$ and $\alpha$ is regular if and only if $\beta$ is regular. 
 
If $n=4$, then we can assume that $c=1$, whence 
$\beta\in\left\{ \smallt{1&2&3&4\\1&1&2&3}, \smallt{1&2&3&4\\1&1&2&4}, \varepsilon_4=\smallt{1&2&3&4\\1&1&3&4}\right\}$. 
Since $\varepsilon_4^2=\varepsilon_4$, 
$\smallt{1&2&3&4\\1&1&2&3}\smallt{1&2&3&4\\1&3&4&4}\smallt{1&2&3&4\\1&1&2&3}=\smallt{1&2&3&4\\1&1&2&3}$,  
$\smallt{1&2&3&4\\1&1&2&4}\smallt{1&2&3&4\\1&3&4&4}\smallt{1&2&3&4\\1&1&2&4}=\smallt{1&2&3&4\\1&1&2&4}$
and $\smallt{1&2&3&4\\1&3&4&4}\in\Sigma_4(\r=3)$, 
then $\beta$ is regular and so $\alpha$ is regular (whether we consider $\alpha\in\Delta_4$ or $\alpha\in\Sigma_4$). 

Now, let us suppose that $n\geqslant5$ and consider the notation and scenarios from the proof of Theorem \ref{deltarank}.  
If $c=1$ and $2\not\in\im(\beta)$, then $\beta\in \B_n$, whence $\beta$ is a group-element, so $\beta$ is regular. 
If $c=1$ and $2\in\im(\beta)$, then $\beta=(\beta\lambda'_i)\lambda_i$ with $\lambda_i,\lambda'_i\in\Sigma_n(\r=n-1)$ and  
$\beta\lambda'_i\in \B_n$, whence $\beta\lambda'_i$ is regular and $\mathscr{R}$-related to $\beta$, so $\beta$ is regular. 
If $c=2$ and $1\not\in\im(\beta)$, then $\beta\in \B'_n$, whence $\beta$ is a group-element and so $\beta$ is regular. 
If $c=2$, $1\in\im(\beta)$ and $i$ is even, then $\beta=(\beta\nu'_i)\nu_i$ with $\nu_i,\nu'_i\in\Sigma_n(\r=n-1)$ and  
$\beta\nu'_i\in \B_n$, whence $\beta\nu'_i$ is regular and $\mathscr{R}$-related to $\beta$, so $\beta$ is regular. 
Let us consider the remaining case, i.e. $c=2$, $1\in\im(\beta)$ and $i$ is odd. 
In this situation in hand, we have $\beta=(\beta\xi'_i)\xi_i$ with $\xi_i,\xi'_i\in\Sigma_n(\r=n-1)$.  
Moreover, $\sigma_0(\beta\xi'_i)\tau_0\in\Sigma_n(\r=n-1)$ falls under the case $c=1$, 
whence $\sigma_0(\beta\xi'_i)\tau_0$ is regular, as previously seen. 
Hence, $\beta\xi'_i$ is regular, since $\sigma_0,\tau_0\in\Gamma_n$, and so $\beta$ is regular, since $\beta\mathscr{R}\beta\xi'_i$. 
Thus, we have shown in all cases that $\beta$ is regular and so $\alpha$ is regular (whether we consider $\alpha\in\Delta_n$ or $\alpha\in\Sigma_n$). 

Next, for $n\geqslant4$, let us take $\alpha\in\Delta_n(\r\leqslant n-2)$. Clearly, there exist $x_0\in\im(\alpha)$ and $y_0\in[n]\setdif\im(\alpha)$ such that 
$x_0+y_0$ is odd. Let $X$ be a transversal of $\ker(\alpha)$. Define $\alpha'\in\T_n$ by: $x\alpha'$ is the element of $x\alpha^{-1}\cap X$, if $x\in\im(\alpha)$, and 
$x\alpha'$ is the element of $x_0\alpha^{-1}\cap X$, if $x\in[n]\setdif\im(\alpha)$. Then, clearly, $\r(\alpha')=\r(\alpha)$ and $\alpha=\alpha\alpha'\alpha$. 
Moreover, $x_0\alpha'=y_0\alpha'$, whence $\alpha'\in\Delta_n(\r\leqslant n-2)$ and so $\alpha$ is regular in $\Delta_n$, as required.  
\end{proof}

We now turn our attention to the monoid $\Sigma_n$. 
Let 
$$
\gamma_n=\transf{1&2&3&4&5&\cdots&n\\1&2&1&2&3&\cdots&n-2}\quad\mbox{for $n\geqslant4$}, 
$$
$$
\gamma'_n=\transf{1&2&3&4&5&6&7&8&\cdots&n\\1&3&1&4&2&5&2&6&\cdots&n-2}\quad\mbox{for $n\geqslant7$}, 
$$
$$
\gamma''_n=\transf{1&2&3&4&5&6&7&8&9&\cdots&n\\3&1&4&1&5&2&6&2&7&\cdots&n-2}\quad\mbox{for $n\geqslant8$}, 
$$
$$
\delta_n=\transf{1&2&3&4&5&6&\cdots&n\\1&2&1&3&1&4&\cdots&n-2}\quad\mbox{for $n\geqslant5$}, 
$$
and 
$$
\delta'_n=\transf{1&2&3&4&5&6&7&\cdots&n\\2&1&3&1&4&1&5&\cdots&n-2}\quad\mbox{for $n\geqslant6$}. 
$$
Clearly, all of these elements belong to $\Sigma_n(\r=n-2)\setdif\Delta_n$. 
Furthermore, the following lemma shows that their kernels, up to multiplication by a unit, are \textit{representatives} of the elements of $\Sigma_n(\r=n-2)\setdif\Delta_n$. 

\begin{lemma}\label{sigma1}
Let $\alpha,\beta\in\Sigma_n(\r=n-2)\setdif\Delta_n$ be such that one of the following two situations occurs: 
\begin{enumerate}
\item $\alpha$ and $\beta$ have two distinct kernel classes of size $2$, 
namely $\{x_1,y_1\}$ and $\{x_2,y_2\}$ of $\alpha$ and $\{x'_1,y'_1\}$ and $\{x'_2,y'_2\}$ of $\beta$,  such that $x_1+x'_1$ and $x_2+x'_2$ are even; 

\item $\alpha$ and $\beta$ have kernel classes $\{x,y,z\}$ and $\{x',y',z'\}$, respectively, of size $3$ such that $x+x'$ is even. 
\end{enumerate}
Then, there exists $\sigma\in\Gamma_n$ such that $\ker(\beta)=\ker(\sigma\alpha)$. 
\end{lemma}
\begin{proof}
Let us first assume that we have $1$.  As $\alpha,\beta\not\in\Delta_n$ and $x_1+x'_1$ and $x_2+x'_2$ are even, 
then $x_1,y_1,x'_1$ and $y'_1$ all have the same parity, and the same is true for $x_2,y_2,x'_2$ and $y'_2$. 
Therefore, there exists $\sigma_0\in\PAP_n^+$ that extends the partial permutation  
$\smallt{x'_1 & y'_1 & x'_2 & y'_2\\x_1 & y_1 & x_2 & y_2}$.  Let $\sigma=\sigma_0$, if $\sigma_0$ is even, and 
$\sigma=\sigma_0(x_1,y_1)$, if $\sigma_0$ is odd. Then, $\sigma\in\A_n$ and so $\sigma\in\Gamma_n$. 
Moreover, clearly, $\ker(\beta)=\ker(\sigma\alpha)$. 
If we have $2$, then $x,y,z,x',y'$ and $z'$ all have the same parity, and the same line of reasoning as in the previous case applies here as well. 
\end{proof}

\begin{theorem}\label{sigmarank}
For $n\geqslant8$,  $\Sigma_n=\langle \Delta_n, \gamma_n,\gamma'_n,\delta_n\rangle$, if $n$ is even, and 
$\Sigma_n=\langle \Delta_n, \gamma_n,\gamma'_n,\gamma''_n,\delta_n,\delta'_n\rangle$, if $n$ is odd. 
In addition,  
$\Sigma_4=\langle \Delta_4, \gamma_4\rangle$, 
$\Sigma_5=\langle \Delta_5, \gamma_5,\delta_5\rangle$,  
$\Sigma_6=\langle \Delta_6, \gamma_6,\delta_6\rangle$ and 
$\Sigma_7=\langle \Delta_7, \gamma_7,\gamma'_7,\delta_7,\delta'_7\rangle$.  
Moreover, the rank of $\Sigma_n$ is 
$$
\left\{
\begin{array}{ll}
3 & \mbox{if $n=4$}\\
6 & \mbox{if $n=5$}\\
5 & \mbox{if $n=6$}\\ 
9 & \mbox{if $n$ is odd and $n\geqslant7$}\\
6 & \mbox{if $n$ is even and $n\geqslant8$}.
\end{array}
\right.
$$
\end{theorem}
\begin{proof}
Let $n\geqslant4$. 
In order to avoid treating small $n$ differently,   let us take, for example, $\gamma'_4=\gamma''_4=\delta_4=\delta'_4=\iota_4$,  
$\gamma'_5=\gamma''_5=\delta'_5=\iota_5$, $\gamma'_6=\gamma''_6=\iota_6$ and $\gamma''_7=\iota_7$. 
Since $\Sigma_n(\r\leqslant n-2)=\T_n(\r\leqslant n-2)=\langle\T_n(\r=n-2)\rangle$ (see e.g. \cite[Lemma 4]{Howie&McFadden:1990}),  
we just need to focus on the elements of  $\Sigma_n(\r=n-2)\setdif\Delta_n$. 

Let $\alpha\in\Sigma_n(\r=n-2)$. 
If $\ker(\alpha)=\ker(\gamma_n)$, then $\alpha=\gamma_n\beta$, with e.g. 
$$
\beta=\transf{1&2&3&\cdots&n-2&n-1&n\\1\alpha&2\alpha&5\alpha&\cdots&n\alpha&1\alpha&1\alpha}\in\Delta_n(\r=n-2),
$$ 
whence $\alpha\in\langle\Delta_n,\gamma_n\rangle$. 
If $\ker(\alpha)=\ker(\gamma'_n)$, then $\alpha=\gamma'_n\beta$, with e.g. 
$$
\beta=\transf{1&2&3&4&5&6&\cdots&n-2&n-1&n\\1\alpha&5\alpha&2\alpha&4\alpha&6\alpha&8\alpha&\cdots&n\alpha&1\alpha&1\alpha}\in\Delta_n(\r=n-2),
$$ 
whence $\alpha\in\langle\Delta_n,\gamma'_n\rangle$. 
If $\ker(\alpha)=\ker(\gamma''_n)$, then $\alpha=\gamma''_n\beta$, with e.g. 
$$
\beta=\transf{1&2&3&4&5&6&7&\cdots&n-2&n-1&n\\2\alpha&6\alpha&1\alpha&3\alpha&5\alpha&7\alpha&9\alpha&\cdots&n\alpha&1\alpha&1\alpha}\in\Delta_n(\r=n-2),
$$ 
whence $\alpha\in\langle\Delta_n,\gamma''_n\rangle$. 
If $\ker(\alpha)=\ker(\delta_n)$, then $\alpha=\delta_n\beta$, with e.g. 
$$
\beta=\transf{1&2&3&4&\cdots&n-2&n-1&n\\1\alpha&2\alpha&4\alpha&6\alpha&\cdots&n\alpha&1\alpha&1\alpha}\in\Delta_n(\r=n-2),
$$ 
whence $\alpha\in\langle\Delta_n,\delta_n\rangle$. 
If $\ker(\alpha)=\ker(\delta'_n)$, then $\alpha=\delta'_n\beta$, with e.g. 
$$
\beta=\transf{1&2&3&4&5&\cdots&n-2&n-1&n\\2\alpha&1\alpha&3\alpha&5\alpha&7\alpha&\cdots&n\alpha&1\alpha&1\alpha}\in\Delta_n(\r=n-2),
$$ 
whence $\alpha\in\langle\Delta_n,\delta'_n\rangle$. 

Now, let $\alpha\in\Sigma_n(\r=n-2)\setdif\Delta_n$. Then, $\alpha$ has exactly either two non-singular kernel classes of size $2$ or 
one non-singular kernel class of size $3$. Hence, by Lemma \ref{sigma1}, 
there exists $\gamma\in\{\gamma_n,\gamma'_n,\gamma''_n,\delta_n,\delta'_n\}$ and $\sigma\in\Gamma_n$ such that $\ker(\gamma)=\ker(\sigma\alpha)$. 
Thus, by the first part of the proof, there exists $\beta\in\Delta_n(\r=n-2)$ such that $\sigma\alpha=\gamma\beta$ and so 
$\alpha=\sigma^{-1}\gamma\beta\in\langle \Delta_n,\gamma\rangle$. 
Therefore, we obtain $\Sigma_n=\langle \Delta_n, \gamma_n,\gamma'_n,\gamma''_n,\delta_n,\delta'_n\rangle$. 

Next, suppose that $n$ is even. Then, $\sigma_n\in\Gamma_n$. 
For $n\geqslant8$, it is clear that $\ker(\gamma'_n)=\ker(\sigma_n\gamma''_n)$, whence $\sigma_n\gamma''_n=\gamma'_n\beta$, for some $\beta\in\Delta_n(\r=n-2)$, and so $\gamma''_n=\sigma_n^{-1}\gamma'_n\beta\in\langle\Delta_n,\gamma'_n\rangle$. 
Similarly, for $n\geqslant6$, $\ker(\delta_n)=\ker(\sigma_n\delta'_n)$, whence $\sigma_n\delta'_n=\delta_n\beta'$, for some $\beta'\in\Delta_n(\r=n-2)$, 
and so $\delta'_n=\sigma_n^{-1}\delta_n\beta'\in\langle\Delta_n,\delta_n\rangle$. 
Therefore, in this case, we get $\Sigma_n=\langle \Delta_n, \gamma_n,\gamma'_n,\delta_n\rangle$. 

At this point, it remains to determine the rank of $\Sigma_n$. 
Clearly, in this context, it is sufficient to calculate how many elements of $\Sigma_n(\r=n-2)\setdif\Delta_n$ we need to add to $\Delta_n$ in order to obtain a set of generators of $\Sigma_n$. 
Let $C$ be a set of generators of $\Sigma_n$. Let $\gamma\in\Sigma_n(\r=n-2)\setdif\Delta_n$. 
Let $\alpha_1,\ldots,\alpha_k\in C$ ($k\in\N$) be such that $\gamma=\alpha_1\cdots\alpha_k$. 
Since $\Sigma_n(\r\geqslant n-1)\subseteq\Delta_n$, there exists $1\leqslant i\leqslant k$ such that $\alpha=\alpha_1\cdots\alpha_{i-1}\in\Delta_n$ 
($\alpha$ can be the identity, e.g. if $i=1$) and $\alpha_i\in\Sigma_n(\r=n-2)\setdif\Delta_n$. 
If $\alpha\not\in\Gamma_n$, then there exist $x,y\in[n]$ such that $x\alpha=y\alpha$ and $x+y$ is odd, 
whence $x\gamma=y\gamma$, which is a contradiction. Thus, $\alpha\in\Gamma_n$. 
As $\ker(\alpha\alpha_i)\subseteq\ker(\gamma)$ and $\r(\alpha\alpha_i)=n-2=\r(\gamma)$, we have $\ker(\alpha\alpha_i)=\ker(\gamma)$. 
First, suppose that $\gamma$ has two distinct kernel classes $\{x_1,y_1\}$ and $\{x_2,y_2\}$ of size $2$. 
Then, $\alpha_i$ also have two distinct kernel classes of size $2$, namely $\{x_1\alpha,y_1\alpha\}$ and $\{x_2\alpha,y_2\alpha\}$. 
Recall that $x_1+y_1$ and $x_2+y_2$ must be even and, for $x\in[n]$,  if $\alpha\in\PAP_n^+$  [respectively, $\alpha\in\PAP_n^-$], 
then $x$ is odd if and only if $x\alpha$ is odd  [respectively, even]. Hence, if $\alpha\in\PAP_n^+$ [respectively, $\alpha\in\PAP_n^-$], then:
\begin{itemize}
\item[(1)] $x_1,y_1,x_2$ and $y_2$ odd ($n\geqslant7$) implies $x_1\alpha,y_1\alpha, x_2\alpha$ and $y_2\alpha$ odd [respectively, even];
\item[(2)] $x_1,y_1,x_2$ and $y_2$ even ($n\geqslant8$) implies $x_1\alpha,y_1\alpha,x_2\alpha$ and $y_2\alpha$ even [respectively, odd];
\item[(3)] $x_1$ and $y_1$ odd and $x_2$ and $y_2$ even ($n\geqslant4$) implies $x_1\alpha$ and $y_1\alpha$ odd [respectively, even] and $x_2\alpha$ and $y_2\alpha$ even [respectively, odd]. 
\end{itemize}
On the other hand, suppose that $\gamma$ has a kernel class $\{x,y,z\}$ of size $3$. 
Then, $x,y$ and $z$ have the same parity and $\{x\alpha,y\alpha,z\alpha\}$ is a kernel class of $\alpha_i$ of size $3$. So, 
if $\alpha\in\PAP_n^+$ [respectively, $\alpha\in\PAP_n^-$], then
\begin{itemize}
\item[(4)] $x,y$ and $z$ odd ($n\geqslant5$) implies $x\alpha,y\alpha$ and $y_2\alpha$ odd [respectively, even];
\item[(5)] $x,y$ and $z$ even ($n\geqslant6$) implies $x\alpha,y\alpha$ and $y_2\alpha$ even [respectively, odd]. 
\end{itemize}
Therefore, if $n$ is odd and $n\geqslant 9$, then $C$ must have at least five different elements of $\Sigma_n(\r=n-2)\setdif\Delta_n$, 
since $\PAP_n^-=\emptyset$. For $n=5$ only situations (3) and (5) are possible, whence $C$ must have at least two different elements of 
$\Sigma_5(\r=4)\setdif\Delta_5$, and for $n=7$ only situation (2) is not possible, whence $C$ must have at least four different elements of $\Sigma_7(\r=5)\setdif\Delta_7$. 
In contrast, if $n$ is even, then $\PAP_n^-\cap\Gamma_n\neq\emptyset$ and, for $n\geqslant8$, 
we can only guarantee that $C$ must have at least three different elements of $\Sigma_n(\r=n-2)\setdif\Delta_n$. 
This number reduces to one, if $n=4$, since only situation (1) is possible, and to two, if $n=6$, since situations (1) and (2) are not possible. 
In light of Theorem \ref{deltarank}, the proof is now complete.
\end{proof}

\section*{Acknowledgment}

This work is funded by national funds through the FCT - Funda\c c\~ao para a Ci\^encia e a Tecnologia, I.P., 
under the scope of the projects UID/00297/2025 (https://doi.org/10.54499/UID/00297/2025) and 
UID/PRR/00297/2025 (https://doi.org/10.54499/UID/PRR/00297/2025) (Center for Mathematics and Applications \-- NOVA Math).  

\section*{Declarations} The author declares no conflicts of interest.

\end{document}